\documentclass[11pt]{article}%
\usepackage{amsmath}
\usepackage{amsfonts}
\usepackage{amssymb}
\usepackage{amsthm}
\usepackage{graphicx}%
\providecommand{\U}[1]{\protect\rule{.1in}{.1in}}
\newtheorem{theorem}{Theorem}[section]
\newtheorem{corollary}[theorem]{Corollary}
\newtheorem{lemma}[theorem]{Lemma}
\newtheorem{proposition}[theorem]{Proposition}
\newtheorem{remark}[theorem]{Remark}

\theoremstyle{definition}

\theoremstyle{remark}

\numberwithin{equation}{section}

\begin{document}

\title{Transmutation operators and a new representation for solutions of perturbed Bessel equations}
\author{Vladislav V. Kravchenko, Sergii M. Torba\\{\small Departamento de Matem\'{a}ticas, CINVESTAV del IPN, Unidad
Quer\'{e}taro, }\\{\small Libramiento Norponiente \#2000, Fracc. Real de Juriquilla,
Quer\'{e}taro, Qro., 76230 MEXICO.}\\
{\small e-mail: vkravchenko@math.cinvestav.edu.mx,
storba@math.cinvestav.edu.mx \thanks{Research was supported by CONACYT, Mexico
via the projects 222478 and 284470. Research of VK was supported by the
Regional mathematical center of the Southern Federal University, Russia.}}}
\maketitle

\begin{abstract}
New representations for an integral kernel of the transmutation operator and for
a regular solution of the perturbed Bessel equation of the form $-u^{\prime\prime}+\left(  \frac{\ell(\ell+1)}{x^{2}}+q(x)\right)
u=\omega^{2}u$ are obtained. The integral kernel is represented as a Fourier-Jacobi series. The solution is represented as a Neumann series of
Bessel functions uniformly convergent with respect to $\omega$. For the
coefficients of the series convenient for numerical computation
recurrent integration formulas are obtained. The new representation improves the ones from \cite{KTC2017} and \cite{KShT2018} for large values of $\omega$ and $\ell$ and for non-integer values of $\ell$.

The results are based on application of several ideas from the classical
transmutation (transformation) operator theory, asymptotic formulas for the solution, results connecting the decay rate of the Fourier transform with the smoothness of a function, the Paley-Wiener theorem and some results from constructive approximation theory.

We show that the analytical representation obtained among other possible
applications offers a simple and efficient numerical method able to compute
large sets of eigendata with a nondeteriorating accuracy.

\end{abstract}

\section{Introduction}

In the present work we consider a second order singular differential equation
\begin{equation}
Lu = -u^{\prime\prime}+\left(  \frac{\ell(\ell+1)}{x^{2}}+q(x)\right)  u=\omega
^{2}u,\qquad x\in(0,b], \label{PertBessel}
\end{equation}
where $\ell$ is a real number, $\ell\geq-\frac{1}{2}$, $q$ is a
complex-valued function satisfying the condition
\[
    x^\mu q(x) \in L_1(0,b)\qquad \text{for some } 0\le \mu<1/2
\]
and $\omega$ is a (complex)
spectral parameter. Equations of the form \eqref{PertBessel} appear naturally in many real-world
applications after a separation of variables and therefore have received
considerable attention (see, e.g., \cite{BoumenirChanane}, \cite{CKT2013},
\cite{CKT2015}, \cite{Chebli1994}, \cite{Guillot 1988}, \cite{KosTesh2011},
\cite[Sect. 3.7]{Okamoto}, \cite{SitnikShishkina}, \cite{Sohin1971}, \cite{Weidmann}).

In \cite{KNT 2015} a new representation for solutions of one-dimensional Schr\"{o}dinger equation (the case $\ell=0$ in \eqref{PertBessel}) in the form of Neumann series of Bessel functions (NSFB) \cite{Baricz et al, Wilkins} was obtained. The representation possesses such remarkable properties as exponentially fast convergence for smooth potentials and uniform error bound for partial sums for all $\omega\in\mathbb{R}$. The idea behind this representation is the existence of the transmutation operator connecting a solution of the equation with the solution of the simpler equation having $q\equiv 0$. A transmutation operator can be realized in the form of a Volterra integral operator. Expanding its integral kernel into a Fourier-Legendre series allowed us to obtain the NSBF representation for the solution.

In \cite{KTC2017} we proposed an NSBF representation for the regular solution $\tilde y(\omega, x)$ of \eqref{PertBessel} normalized by its asymptotics $\tilde y(\omega, x)\sim x^{\ell+1}$ at zero. However, the above mentioned advantages of the regular case were lost for the representation from \cite{KTC2017}. The reasons are the following. First, for any fixed $x>0$ the function $\tilde y(\omega, x)$ decays as $\frac{1}{\omega^{\ell+1}}$ as $\omega \to \infty$. For that reason any error bound uniform with respect to $\omega\in\mathbb{R}$ is useful only in some neighborhood of $\omega=0$, and for large values of the parameter $\ell$ this neighborhood becomes rather small. Second, the exponential convergence of the partial sums was lost for $\ell\not\in\mathbb{N}$, the convergence was bounded by $\frac{1}{N^{2\ell+3}}$ (here $N$ is the truncation parameter) independently of the smoothness of the potential. The reason was in the underlying Mehler's integral representation for the solution $\tilde y$:
\begin{equation*}
\begin{split}
\tilde y(\omega,x)&=\frac{2^{\ell+1}\Gamma\left(  \ell+\frac{3}{2}\right)  }{\sqrt{\pi
}\omega^{\ell}}xj_{\ell}\left(  \omega x\right)  +\int_{-x}^{x}\tilde R(x,t)e^{i\omega
t}\,dt\\
&=\frac{2^{\ell+1}\Gamma\left(  \ell+\frac{3}{2}\right)  }{\sqrt{\pi
}\omega^{\ell}}xj_{\ell}\left(  \omega x\right)  +\int_{0}^{x}R(x,t)\cos\omega
t\,dt,
\end{split}
\end{equation*}
here $R = 2\tilde R$ and is an even function, and $j_\ell$ denotes the spherical Bessel function of the first kind, see \cite[Section 10.1]{Abramowitz}.
The integral kernel $\tilde R$ admits the following representation
\begin{equation}\label{R first term}
    \tilde R(x,t) = c(x) \left(1-\frac{t^2}{x^2}\right)^{\ell+1} + \tilde R_1(x,t),
\end{equation}
where $c(x) = 2^{\ell-3/2}\Gamma\left(\ell+\frac 32\right)\int_0^x q(s)\,ds$ and  $\tilde R_1$, extended by 0 onto the whole line, possesses at least $[\ell+2]$ derivatives as a function of $t$. As a result, for non-integer values of $\ell$ the decay rate of the Fourier-Legendre coefficients of $R$ is polynomial and is determined by the first term in \eqref{R first term}, independent on the smoothness of $q$.

To overcome the first problem, one can use the transmutation operator for the perturbed Bessel equation. Recall that a transmutation operator intertwining \eqref{PertBessel} with the unperturbed Bessel equation
\begin{equation}\label{UnpertBessel}
    -y^{\prime\prime}+\frac{\ell(\ell+1)}{x^{2}} y=\omega
^{2}y
\end{equation}
has the form \cite{Volk, Sta2, CoudrayCoz, Trimeche, Holz2020}
\begin{equation}\label{Transmut}
    u(\omega, x) = T[y(\omega, x)] := y(\omega, x) + \int_0^x K(x,t)y(\omega, t)\,dt.
\end{equation}
In such case, if one takes $y(\omega, x) = \omega x j_\ell(\omega x)$ as a regular solution of \eqref{UnpertBessel} (note that this solution remains bounded and does not decay as $\omega\to\infty$, see \cite[(9.2.1)]{Abramowitz}), finds an approximation $K_N$ for $K$ satisfying $\|K(x,\cdot) - K_N(x,\cdot)\|_{L_2(0,x)}\le \varepsilon_N(x)$ and defines an approximate solution $u_N$ by using this $K_N$ in \eqref{Transmut}, one easily obtains (applying the Cauchy-Schwartz inequality) that
\begin{equation*}
    |u(\omega, x)-u_N(\omega, x)|\le c\varepsilon_N(x)
\end{equation*}
uniformly for $\omega\in\mathbb{R}$. And since $u(\omega, x)$ does not decay as $\omega\to\infty$, such approximation is useful for both small and large values of $\omega$.

However, the problem of constructing an approximation $K_N$ for which at the same time the integrals in \eqref{Transmut} can be easily evaluated and rapid decay of the error bound $\varepsilon_N$ as $N\to\infty$ can be proved, is not an easy task, see, e.g., \cite{KrST} for an attempt. The integral kernel $K$ is a solution of a singular Goursat problem which can be transformed into one of the several equivalent integral equations, see \cite{Volk}, \cite{CoudrayCoz}, \cite{Holz2020}. We do not expect that any of these integral equations can provide an efficient approximation method for the integral kernel. The main reason is the following: we conjecture that the integral kernel has the form
\[
K(x,t) = \frac{t^{\ell+1}}{x^\ell}\widetilde K(x, x^2-t^2),
\]
where $\widetilde K$ is a nice function (in the sense that it is at least $C^\infty$ in the second variable for sufficiently smooth potentials). The base for such conjecture is the formula (4.6) from \cite{Chebli1994}
\[
K(x,t) = \frac{t^{\ell+1}}{x^\ell} \sum_{p=1}^\infty \frac{B_p(x)}{2^{p-1}\Gamma(p)}(x^2-t^2)^{p-1},
\]
proved there under the assumption that $q$ possesses a holomorphic extension onto the disk of radius $2xe\sqrt{1+2|\ell|}$. We believe that the factor $\frac{t^{\ell+1}}{x^\ell}$ plays a crucial role in the construction of a good approximation to $K$, however it gets strongly obscured in the known integral equations for $K$.

In \cite{KShT2018} we proposed to use a series representation for the integral kernel $R$ and an Erdelyi-Kober operator to obtain a series representation for the integral kernel $K$, and for integer values of $\ell$ the partial sums provided a good approximation for $K$. For non-integer values of $\ell$, the approximation obtained was not practical to substitute it into \eqref{Transmut}.

In the present paper we show that the Legendre polynomials are not the best choice for representing the kernel $R$, and one needs to use the Jacobi polynomials instead. The motivation for this claim is the formula \eqref{R first term} and the following representation for $R$
\begin{equation*}\label{R series Chebli}
    R(x,t) = (x^2-t^2)^{\ell+1} \sum_{p=1}^\infty \frac{x^p B_p(x)}{\sqrt\pi 2^{\ell+p-1/2} \Gamma(\ell+p)}(x^2-t^2)^{p-1},
\end{equation*}
which follows from formula (4.4) from \cite{Chebli1994} and is proved there under the assumption that $q$ possesses a holomorphic extension onto the disk of radius $2xe\sqrt{1+2|\ell|}$.

We show that the factor $\bigl(1-(t/x)^2\bigr)^{\ell+1}$ should be used as a weight for the Fourier-Jacobi expansion, and $\tilde R$ can be represented as
\begin{equation}\label{R Jacobi}
    \tilde R(x,t) = \left(1-\frac{t^2}{x^2}\right)^{\ell+1}\sum_{n=0}^\infty \frac{\tilde\beta_n(x)}{x} P_{2n}^{(\ell+1,\ell+1)}\left(\frac tx\right),
\end{equation}
where $P_n^{(\alpha,\beta)}$ stands for the Jacobi polynomials, leading to the following representation for the integral kernel $K$
\begin{equation}\label{K Jacobi}
\begin{split}
    K(x,t) &= \frac{2\sqrt\pi }{x^{2\ell+3}\Gamma(\ell+3/2)}\sum_{n=0}^\infty \frac{(-1)^n \tilde\beta_n(x)\Gamma(\ell+2n+2)}{(2n)!}
    t^{\ell+1}P_n^{(\ell+1/2,0)}\left(1-\frac{2t^2}{x^2}\right)\\
    &= \sum_{n=0}^\infty \frac{\beta_n(x)}{x^{\ell+2}} t^{\ell+1}P_n^{(\ell+1/2,0)}\left(1-\frac{2t^2}{x^2}\right),
\end{split}
\end{equation}
where we denoted
\begin{equation}\label{beta and tilde beta}
    \beta_n(x) = (-1)^{n} \frac{2\sqrt{\pi} \Gamma(\ell+2n+2)}{x^{\ell+1}\Gamma(\ell+3/2)(2n)!}\tilde \beta_n(x).
\end{equation}

As a result, the following representation for the regular solution of \eqref{PertBessel} is obtained
\begin{equation}\label{u NSBF}
    u(\omega, x) = \omega x j_\ell(\omega x) + \sum_{n=0}^\infty \beta_n(x) j_{\ell+2n+1}(\omega x).
\end{equation}
We show the uniform with respect to $\omega\in\mathbb{R}$ convergence rate estimates for \eqref{u NSBF}, prove faster than polynomial convergence for $C^\infty$ potentials and present efficient for numerical implementation recurrent formulas to calculate the coefficients $\beta_n$ in which only integration is used, and no differentiation is required.

For the derivative of the regular solution $u(\omega, x)$ we obtain the following representation
\begin{equation}\label{uprime NSBF}
    u'(\omega, x) = \omega^2 x j_{\ell-1}(\omega x) + \left(\frac{xQ(x)}2 - \ell\right)\omega j_\ell(\omega x) +
     \sum_{n=0}^\infty \gamma_n(x) j_{\ell+2n+1}(\omega x),
\end{equation}
possessing all the remarkable properties of \eqref{u NSBF}: uniform with respect to $\omega\in\mathbb{R}$ error bounds for truncated series, faster than polynomial convergence for $C^\infty$ potentials and efficient for numerical implementation formulas to calculate the coefficients $\gamma_n$.

The representation \eqref{K Jacobi} for the integral kernel $K$ together with the decay rate estimates for the coefficients $\beta_n$ can be of the independent interest for studying the transmutation operator for perturbed Bessel equation and for solving the inverse spectral problem like it was done in the regular case in \cite{Kravchenko2019}, \cite{DKK2019}, \cite{KravchenkoInverseBook}.

We would like to mention that despite of lots of technical details in proofs, application of the representations \eqref{u NSBF} and \eqref{uprime NSBF} for numerical solution of boundary value or spectral problems for equation \eqref{PertBessel} is very simple. All that one needs is a solution $u_0$ of the equation
\[
-u_0'' + \left(\frac{\ell(\ell+1)}{x^2} + q(x)\right) u_0=0,
\]
and its derivative $u_0'$ (both can be obtained numerically). Then one calculates the coefficients $\{\beta_n\}_{n=0}^{N_1}$ and $\{\gamma_n\}_{n=0}^{N_2}$ using \eqref{etan}--\eqref{beta_n alt} and \eqref{gamma_n alt}. To estimate optimal values of $N_1$ and $N_2$ one
utilizes the equalities \eqref{Verification beta} and \eqref{Verification gamma}. Computation of both $u(\omega, x)$ and $u'(\omega, x)$ for each fixed $\omega$ reduces to the computation of the values of several Bessel functions. As a result, for example, several hundreds eigenvalues of a Sturm-Liouville problem can be obtained in less than a second.

The paper is organized as follows. In Section \ref{Sect2} we present some preliminary information about transmutation operators and asymptotic estimates of the solutions. In Section \ref{Sect3} we study the smoothness of the integral kernel $\tilde R$ (Propositions \ref{Prop R near pm x} and \ref{Prop Smoothness R}) and convergence rate of its Fourier-Jacobi expansion (Theorem \ref{Thm Convergence R} and Proposition \ref{Prop R pointwise}). In Section \ref{Sect4} we show that the integral kernel $K$ possesses Fourier-Jacobi representation \eqref{K Jacobi} (Theorem \ref{Thm K repr}), extend this representation to a wider class of potentials and prove its convergence rate (Theorem \ref{Thm Convergence K}). In Section \ref{Sect5} we obtain the representation \eqref{u NSBF} (Theorem \ref{thm Rep Sol}). In Section \ref{Sect6} we present necessary facts about the derivative of the regular solution, its relation with the transmutation operator and prove the representation \eqref{uprime NSBF} (Theorems \ref{Thm K1 repr} and \ref{thm Rep DerSol}). In Section \ref{Sect7} we obtain the recurrent formulas for the coefficients $\beta_n$ and $\gamma_n$. In Section \ref{Sect8} we present numerical results for solution of two spectral problems and study the decay of the coefficients $\beta_n$. In Appendix \ref{AppB} we prove that the transmutation integral kernel depends continuously on the potential (Corollary \ref{Corr continuity K}). Finally, in Appendix \ref{AppA} we prove that the regular solution of equation \eqref{PertBessel} is non-vanishing on the whole $(0,b]$ for any sufficiently large negative $\lambda = \omega^2$ (Proposition \ref{Prop NonVanishing Sol}).

\section{Transmutation operator, solution asymptotics and Mehler-type integral representation}
\label{Sect2}

Denote
\begin{equation}\label{Sol bl}
    b_\ell(\omega x) := \omega x j_\ell(\omega x).
\end{equation}
This function is a regular solution of equation \eqref{UnpertBessel} satisfying the asymptotic condition
\begin{equation}\label{bl asympt}
b_\ell(\omega x)\sim \frac{\sqrt\pi}{2^{\ell+1}\Gamma(\ell+3/2)}(\omega x)^{\ell+1},\qquad x\to 0.
\end{equation}
By $u(\omega,x)$ we denote the regular solution of \eqref{PertBessel} satisfying the same asymptotic condition at 0.

Assume that $q\in C[0,b]$. Then \cite{Volk}, \cite{Sta2}, \cite{CoudrayCoz}, \cite{Holz2020} there exists a unique continuous kernel $K(x,t)$ such that for all $\omega \in\mathbb{C}$
\begin{equation}\label{VolkTransmute}
    u(\omega, x) = T[b_\ell(\omega x)] = b_\ell(\omega x) + \int_0^x K(x,t) b_\ell(\omega t)\,dt.
\end{equation}
The integral kernel $K$ satisfies
\begin{equation}\label{K Goursat}
    K(x,x) = \frac{Q(x)}{2},\qquad \text{where}\quad Q(x):= \int_0^x q(t)\,dt.
\end{equation}
The operator $T$ is called the transmutation operator.

The existence of the transmutation operator \eqref{VolkTransmute} can be established for wider class of potentials.
For the purposes of this work we need that the integral kernel $K$ as a function of $t$ belongs to $L_2(0,x)$ for each fixed $x$.
Let us introduce the notation \cite{KosSakhTesh2010}
\begin{equation}\label{tilde q}
\tilde q(x) = \begin{cases}
q(x), & \ell>-1/2,\\
\bigl(1-\log(x/b)\bigr)q(x),& \ell=-1/2.
\end{cases}
\end{equation}
Then the following condition on $q$ is sufficient for the integral kernel $K$ to belong to $L_2(0,x)$:
\begin{equation}\label{Cond on q}
    x^\mu \tilde q(x) \in L_1(0,b)\qquad \text{for some } 0\le \mu<1/2.
\end{equation}
For the proof we refer the reader to \cite[\S2]{Sta2}, the case of non-integer values of $\ell$ is also covered taking \cite{Griffith1955} into account.

The difference between the solutions $u(\omega,x)$ and $b_\ell(\omega x)$ satisfies the inequality \cite[(2.18)]{KosSakhTesh2010}
\begin{equation*}
    |u(\omega,x) - b_\ell(\omega x)|\le C\left(\frac{|\omega|x}{b+|\omega|x}\right)^{\ell+1} e^{|\operatorname{Im}\omega| x} \int_0^x \frac{y|\tilde q(y)|}{b+|\omega|y} \,dy,
\end{equation*}
from which it follows for the potential $q$ satisfying condition \eqref{Cond on q} that
\begin{equation}\label{SolAsymptot1}
    |u(\omega, x) - b_\ell(\omega x)|\le \frac{\tilde C}{|\omega|^{1-\mu}}, \qquad \omega\in\mathbb{R}.
\end{equation}

More can be said about the asymptotic behavior of the solution as $\omega\to\infty$ if the potential $q$ possesses several derivatives on the whole segment $[0,b]$. Indeed, similarly to the proof of Proposition 4.5 from \cite{KTC2017} (see also \cite[Theorem 4]{FitouhiHamza}) one can see that if $q\in W_1^{2p-1}[0,b]$, $p\in\mathbb{N}$, i.e., $q$ possesses $2p-1$ derivatives, the last one belonging to $L_1(0,b)$, then
\begin{equation}\label{SolMoreTerms}
    u(\omega, x) = b_\ell(\omega x) + \sum_{k=1}^p \frac{ A_k(x)}{2^{\ell+1/2}\Gamma(\ell+3/2)} \frac{\omega x j_{\ell+k}(\omega x)}{\omega ^k}+ \mathcal{R}_p(\omega, x),
\end{equation}
where
\begin{equation}\label{SolAsymptot2}
    |\mathcal{R}_p(\omega, x)|\le\frac{c(\ell,m)}{|\omega|^{p+1}}
\end{equation}
and $\mathcal{R}_p$ as a function of $\omega$ is an even entire function of exponential type $x$.

Application of the Paley-Wiener theorem leads to the following Mehler-type integral representation for the solution $u(\omega,x)$, see \cite{KTC2017}
\begin{equation}\label{Mehler representation}
    d_\ell(\omega) u(\omega, x) = d_\ell(\omega) b_\ell(\omega x) + \int_{-x}^x \tilde R(x,t) e^{i\omega t}\,dt,
\end{equation}
where
\[
d_\ell(\omega) = \frac{2^{\ell+1}\Gamma(\ell+3/2)}{\sqrt\pi \omega^{\ell+1}}.
\]
By $W_2^\alpha(\mathbb{R})$, $\alpha\ge 0$ we denote the fractional-order Sobolev space, also called Bessel potential space \cite[Chap. 7]{Adams} consisting of functions satisfying $f\in L_2(\mathbb{R})$ and $(1+|\xi|^2)^{\alpha/2}\mathcal{F}[f](\xi)\in L_2(\mathbb{R})$, where $\mathcal{F}$ is the Fourier transform operator. Then, extending the integral kernel $\tilde R$ as a function of $t$ by 0 outside of $[-x,x]$ and denoting the resulting function again by $\tilde R$ we have the following \cite[Proposition 4.1]{KTC2017}. $\tilde R$ is a continuous, even, compactly supported on $[-x,x]$ function and such that $\tilde R\in W_2^{\ell+3/2-\mu-\varepsilon}(\mathbb{R})$ for any sufficiently small $\varepsilon>0$.


Additionally, if $q\in W_1^{2p-1}[0,b]$, then \cite[(4.24)]{KTC2017}
\begin{equation}\label{Improvement of R}
    \tilde R(x,t) = \tilde R_{p}(x,t)+\sum_{k=1}^p \frac{A_k(x)x^{\ell+k}}{\sqrt\pi 2^{2\ell+k+1}\Gamma(\ell+k+1)\Gamma(\ell+3/2)}\cdot\left(1-\frac{t^2}{x^2}\right)^{\ell+k},\quad -x\le t\le x,
\end{equation}
where $\tilde R_p$ (extended by 0 outside of $[-x,x]$) is a continuous, even function satisfying $\tilde R_p\in W_2^{\ell+p+3/2-\varepsilon}(\mathbb{R})$  for any sufficiently small $\varepsilon>0$.

The integral kernels $\tilde R$ and $K$ are related by the following relations \cite[(2.5) and (2.6)]{KShT2018}
\begin{equation}
\tilde R(x,s) = \frac{\Gamma \left( \ell+\frac{3}{2}\right) }{\sqrt{\pi }\Gamma (\ell+1)}%
\int_{s}^{x}K(x,t)t^{-\ell}(t^{2}-s^{2})^{\ell}\,dt,  \label{RviaK}
\end{equation}%
and
\begin{equation}
K(x,t)=\frac{4\sqrt{\pi }}{\Gamma \left( \ell+\frac{3}{2}\right) }\frac{t^{\ell+1}}{\Gamma (n-\ell-1)} \left( -\frac{d}{2tdt}\right) ^{n}\int_{t}^{x}(s^{2}-t^{2})^{n-\ell-2}s\tilde R(x,s)ds,  \label{KviaR}
\end{equation}
here $n$ is an arbitrary integer satisfying $n>\ell+1$.

\section{Fourier-Jacobi expansion of the integral kernel $\tilde R$}
\label{Sect3}
In this section we show that the integral kernel $\tilde R$ can be expanded into a Fourier-Jacobi series  \eqref{R Jacobi}, study error bounds for the remainder of the truncated series and the decay rate of the coefficients $\beta_n$.

We recall some definitions from the approximation theory used in this section. By $[[x]]$ we denote the largest integer smaller than $x$, and let $\{\{x\}\} := x-[[x]]$. Then $\{\{x\}\}\in (0,1]$.

Following \cite{ST1978} we say that a function $f$ is Lipschitz of an order $\alpha>0$ on $I$ (may be a segment or the whole line) if
\begin{enumerate}
\item there exist the derivatives of $f$ of all orders up to the order $[[\alpha]]$;
\item $f^{(m)}\in L_\infty(I)$ for all $m\le [[\alpha]]$;
\item $f^{([[\alpha]])}$ satisfies the Lipschitz condition of order $\{\{\alpha\}\}$, i.e., there exists a positive constant $A$ such that
\[
|f^{([[\alpha]])}(x) - f^{([[\alpha]])}(y)|\le A|x-y|^{\{\{\alpha\}\}}\qquad \forall x,y\in I.
\]
\end{enumerate}
We will denote the class of such functions by $\operatorname{Lip}(\alpha,I)$.

Following \cite{Moricz2008} we say that a function $f$ belongs to the Zygmund class $\operatorname{Zyg}(\alpha, I)$ on some $I$ for some $\alpha>0$ if
\begin{enumerate}
\item $f^{(m)}\in C(I)$ for all $m\le [[\alpha]]$;
\item $f^{([[\alpha]])}$ satisfies the Zygmund condition for some constant $C$, i.e.,
\[
|f^{([[\alpha]])}(x+h) - 2f^{([[\alpha]])}(x)+f^{([[\alpha]])}(x-h)|\le Ch^{\{\{\alpha\}\}}\qquad \forall x,h:\ \{x-h,x,x+h\}\subset I.
\]
\end{enumerate}
Note that $\operatorname{Lip}(\alpha, I)\subset \operatorname{Zyg}(\alpha,I)$.

The following proposition immediately follows from \cite[Theorems 1 and 3]{Moricz2008} and shows the relationship between the decay rate of the Fourier transform and the smoothness of a function.
\begin{proposition}\label{Prop LipZyg}
Suppose $f:\mathbb{R}\to\mathbb{C}$ is such that $f\in L^1\cap C(\mathbb{R})$ and its Fourier transform satisfies for some $\alpha>0$
\[
\bigl|\hat f(\xi)\bigr| \le \frac C{|\xi|^{\alpha+1}}\qquad\text{for all }\xi\ne 0.
\]
Then
\[
f\in\operatorname{Lip}(\alpha,\mathbb{R}),\qquad\text{if }\alpha\not\in\mathbb{N},
\]
and
\[
f\in\operatorname{Zyg}(\alpha,\mathbb{R}),\qquad\text{if }\alpha\in\mathbb{N}.
\]
Moreover, for any $m\in\mathbb{N}$, $m\le[[\alpha]]$,
\begin{align*}
f^{(m)}&\in\operatorname{Lip}(\alpha-m,\mathbb{R}),\quad\text{if }\alpha\not\in\mathbb{N},\\
f^{(m)}&\in\operatorname{Zyg}(\alpha-m,\mathbb{R}),\quad\text{if }\alpha\in\mathbb{N}.
\end{align*}
\end{proposition}

Following \cite{Ky1976} we introduce the following notations and definitions. Let $\mathcal{P}_n$ be the set of algebraic polynomials of degree not greater than $n$. Let $W_\alpha(x) = (1-x^2)^{\alpha/2}$, $x\in[-1,1]$, $\alpha> -1/2$. We define as best weighted polynomial approximation of a function $f$ such that $fW_\alpha\in L_2(-1,1)$ the quantity
\begin{equation*}
    E_n(W_\alpha;f) = \inf_{p_n\in \mathcal{P}_n} \| (f-p_n)W_\alpha\|_{L_2(-1,1)},\qquad n=0,1,2,\ldots
\end{equation*}

Denote by $S(\alpha;f,x)$ the orthonormal expansion of $f$ with respect to the system of normalized Jacobi polynomials $\{\tilde P_n^{(\alpha,\alpha)}(x)\}_{n=0}^\infty$, that is
\begin{equation}\label{FourierJacobiSeries}
    f(x)\sim S(\alpha; f,x) = \sum_{k=0}^\infty c_k(\alpha;f) \tilde P_k^{(\alpha,\alpha)}(x),
\end{equation}
where
\begin{equation*}
    c_k(\alpha;f) = \int_{-1}^1 f(x)\tilde P_k^{(\alpha,\alpha)}(x) W_\alpha^2(x)\,dx,\qquad k=0,1,2,\ldots
\end{equation*}
Then the best weighted polynomial approximation of the function $f$ coincides with the norm of the remainder of its Fourier-Jacobi expansion:
\begin{equation}\label{FourierJacobiRemainder}
    E_n(W_\alpha;f) = \left\{\sum_{k=n+1}^\infty |c_k(\alpha;f)|^2\right\}^{1/2}.
\end{equation}

Let $S_k^{(\alpha)}$, $k=1,2,\ldots$ be the set of functions $f$ satisfying
\begin{enumerate}
\item $f$ is a $k$-times iterated integral of $f^{(k)}$ in $(-1,1)$;
\item $f^{(l)}W_{\alpha+l}\in L_2(-1,1)$, $l=0,1,\ldots,k$.
\end{enumerate}
By $S_0^{(\alpha)}$ we define the set of functions $f$ satisfying $W_\alpha f\in L_2(-1,1)$.

Consider the following modulus of continuity:
\begin{align}
  \omega(W_\alpha;f;\delta) & = \sup_{0\le t\le \delta}\left\{\int_0^{5\pi/8} |f^\ast(\theta+t)-f^\ast(\theta)|^2 W_\alpha^{\ast 2}(\theta)\sin\theta\,d\theta\right\}^{1/2}\nonumber
  \displaybreak[2]
  \\
  &\quad +\sup_{0\le t\le \delta}\left\{\int_{3\pi/8}^\pi |f^\ast(\theta-t)-f^\ast(\theta)|^2 W_\alpha^{\ast 2}(\theta)\sin\theta\,d\theta\right\}^{1/2},\quad 0<\delta\le \frac\pi 3, \label{Cont modulus}
  \displaybreak[2]
\end{align}
where $f^\ast(\theta)$ is defined by $f^\ast(\theta):=f(\cos\theta)$, $0\le \theta\le \pi$.

Then the following result holds
\begin{theorem}[\cite{Ky1976}]\label{Thm Approx}
Let $f\in S_k^{(\alpha)}$ for some $k\in\mathbb{N}_0$. Then
\begin{equation*}
    E_{n+k}(W_\alpha;f)\le\frac{c(\alpha,k)}{n^k}\omega\left(W_{\alpha+k}; f^{(k)};\frac 1n\right) \le \frac{c_1(\alpha,k)}{n^k}\|W_{\alpha+k}f^{(k)}\|_{L_2(-1,1)},\qquad n=1,2,\ldots
\end{equation*}
\end{theorem}

In the rest of this section we apply Proposition \ref{Prop LipZyg} and Theorem \ref{Thm Approx} to estimate the decay rate of the coefficients and of the remainder of the Fourier-Jacobi expansion of the integral kernel $\tilde R$.

\begin{proposition}\label{Prop R near pm x}
Let $q$ satisfy condition \eqref{Cond on q} and $x>0$ be fixed. Let the integral kernel $\tilde R$ from \eqref{Mehler representation} as a function of $t$ be extended by $0$ outside of $[-x,x]$. For the sake of simplicity we denote this extended function by the same letter $\tilde R$. Then
\begin{enumerate}
\item $\tilde R(x,\cdot) \in \operatorname{Zyg}(\ell-\mu+1,\mathbb{R})$; moreover, if $\ell-\mu\not\in\mathbb{Z}$, then $\tilde R(x,\cdot) \in \operatorname{Lip}(\ell-\mu+1,\mathbb{R})$;
\item there exist functions $\{r_m\}_{m=0}^{[[\ell-\mu+1]]}$, bounded on $[-x,x]$ and continuous on $(-x,x)$, such that
\begin{equation*}
    \partial^m_tR(x,t) = (x^2-t^2)^{\ell-\mu+1-m} r_m(t),\qquad t\in [-x,x].
\end{equation*}
\end{enumerate}
If additionally $q\in W_1^{2p-1}[0,b]$, then the above statements hold for the integral kernel $\tilde R_p$ from \eqref{Improvement of R} with the change of $\ell-\mu$ by $\ell+p$ in all the formulas: $\tilde R_p(x,\cdot) \in \operatorname{Zyg}(\ell+p+1,\mathbb{R})$, etc.
\end{proposition}

\begin{proof}
Note that by \eqref{Mehler representation} the function $\tilde R$ is the Fourier transform of the function $d_\ell(\omega)[u(\omega,x) - b_\ell(\omega x)]$, which is continuous on $\mathbb{R}$ and decays as $\frac{1}{|\omega|^{\ell-\mu+2}}$ when $\omega\to\infty$ due to \eqref{SolAsymptot1} (see also the proof of Theorem 4.1 from \cite{KTC2017}). Hence the first statement follows immediately from Proposition \ref{Prop LipZyg}.

To prove the second part, consider Taylor's formula for the function $\tilde R$ at the points $t=x$ and $t=-x$. Since the function $\tilde R$ is compactly supported on $[-x,x]$ and continuous together with its derivatives up to the order $[[\ell-\mu+1]]$ on the whole $\mathbb{R}$, we have
\[
\tilde R(x,\pm x) =  \partial_t \tilde R(x,\pm x) =\ldots = \partial_t^{([[\ell-\mu+1]])}\tilde R(x,\pm x) = 0,
\]
hence (for $0\le t\le x$)
\begin{equation}\label{Taylor for R}
\tilde R(x,t) = \frac{(x-t)^{[[\ell-\mu+1]]}}{([[\ell-\mu+1]])!} \partial_t^{[[\ell-\mu+1]]} \tilde R(x,\xi)
\end{equation}
for some $\xi\in (t,x)$.

Now, $\partial_t^{[[\ell-\mu+1]]} \tilde R$ satisfies Zygmund's condition on the whole $\mathbb{R}$. Hence choosing the points $\xi$, $x$ and $2x-\xi$ and taking into account that $\partial_t^{[[\ell-\mu+1]]} \tilde R(x,x)=\partial_t^{[[\ell-\mu+1]]} \tilde R(x,2x-\xi)=0$, we obtain
\begin{equation}\label{Taylor R Zygmund}
    |\partial_t^{[[\ell-\mu+1]]} \tilde R(x,\xi)|\le C|x-\xi|^{\{\{\ell-\mu+1\}\}}\le C|x-t|^{\{\{\ell-\mu+1\}\}}.
\end{equation}
It follows from \eqref{Taylor for R} and \eqref{Taylor R Zygmund} that for $0\le t\le x$
\[
|\tilde R(x,t)|\le C|x-t|^{\ell-\mu+1}.
\]
Now,
\[
\frac{|\tilde R(x,t)|}{(x^2-t^2)^{\ell-\mu+1}}= \frac{|\tilde R(x,t)|}{(x-t)^{\ell-\mu+1}(x+t)^{\ell-\mu+1}}\le \frac{|\tilde R(x,t)|}{x^{\ell-\mu+1}(x-t)^{\ell-\mu+1}}\le \frac{C}{x^{\ell-\mu+1}}.
\]
The proofs for $-x\le t\le 0$ as well as for $m=1,\ldots,[[\ell+1]]$ are similar.

The proof for the case $q\in W_1^{2p-1}[0,b]$ is completely similar if one uses \eqref{SolAsymptot2}.
\end{proof}

Let $x>0$ be fixed. Consider the functions
\begin{equation*}
    g(z):=\tilde R(x,zx) \qquad \text{and}\qquad h(z):=\frac{g(z)}{(1-z^2)^{\ell+1}},\qquad z\in (-1,1).
\end{equation*}
The expansion \eqref{R Jacobi} reduces to the Fourier-Jacobi expansion of the function $h$.
\begin{proposition}\label{Prop Smoothness R}
Let $q$ satisfy condition \eqref{Cond on q}. Then
\[
h\in S^{(\ell+1)}_{[[\ell-\mu+1]]}.
\]
If additionally $q\in W_1^{2p-1}[0,b]$, then $h\in S^{(\ell+1)}_{[[\ell+p+1]]}$.
\end{proposition}

\begin{proof}
Note that the derivative of order $m$, $m\le [[\ell-\mu+1]]$, of the function $h$ can be written as
\begin{equation}\label{h derivative}
    h^{(m)}(z) = \sum_{k=0}^m \frac{g^{(m-k)}(z)\cdot p_k(z)}{(1-z^2)^{\ell+1+k}},
\end{equation}
where $p_k$ are some polynomials in $z$.

It is sufficient to prove that each term in \eqref{h derivative}, multiplied by $W_{\ell+1+m}$, belongs to $L_2(-1,1)$. We have
\begin{equation}\label{h derivative weight}
    \left|\frac{g^{(m-k)}(z)\cdot p_k(z)}{(1-z^2)^{\ell+1+k}}\cdot (1-z^2)^{(\ell+1+m)/2}\right|^2 = \frac{|g^{(m-k)}(z)|^2\cdot |p_k(z)|^2}{(1-z^2)^{\ell+1+2k - m}}.
\end{equation}
Due to Proposition \ref{Prop R near pm x}
\begin{align*}
g^{(m-k)}(z) & = x^{m-k}(x^2-x^2z^2)^{\ell-\mu+1-m+k}r_{m-k}(zx) \\
&= x^{2\ell-2\mu+2-m+k}(1-z^2)^{\ell-\mu+1-m+k}r_{m-k}(zx),
\end{align*}
and \eqref{h derivative weight} reduces to
\[
\frac{x^{4\ell-4\mu+4-2m+2k}|r_{m-k}(zx)|^2|p_k(z)|^2}{(1-z^2)^{m-\ell+2\mu-1}},
\]
integrable over $(-1,1)$ since $m<\ell-\mu+3/2$ and $\mu<1/2$.

For the proof in the case $q\in W_1^{2p-1}[0,b]$ one applies a completely similar reasoning for the function
\[
h_p(z) := \frac{\tilde R_p(x,xz)}{(1-z^2)^{\ell+1}}, \qquad -1<z<1
\]
and notes that due to \eqref{Improvement of R} the functions $h$ and $h_p$ differ by a polynomial in $z$ and hence belong to the same class $S^{(\ell+1)}_{[[\ell+p+1]]}$.
\end{proof}

As it follows from Proposition \ref{Prop Smoothness R}, the function $h$ always belongs at least to $S_0^{(\ell+1)}$. So we may expand it into a Fourier-Jacobi series \eqref{FourierJacobiSeries}. Returning to the integral kernel $\tilde R$ we have
\begin{equation}\label{R series1}
    \frac{\tilde R(x,t)}{\left(1-\frac{t^2}{x^2}\right)^{\ell+1}} = \sum_{k=0}^\infty c_k(\ell+1; h) \tilde P_k^{(\ell+1,\ell+1)}\left(\frac tx\right),
\end{equation}
where
\begin{align}
  c_k(\ell+1;h) &= \int_{-1}^1 h(z)  \tilde P_k^{(\ell+1,\ell+1)}(z)\cdot (1-z^2)^{\ell+1}\,dz  \nonumber \\
  \displaybreak[2]
   &= \frac 1x\int_{-x}^x \tilde R(x,t)   \tilde P_k^{(\ell+1,\ell+1)}\left(\frac tx\right) \,dt,\quad k=0,1,\ldots \label{R coeff1}
\end{align}
The function $\tilde R$ is even, so all coefficients $c_{2k+1}(\ell+1;h)\equiv 0$ and we obtain the representation \eqref{R Jacobi}. The following theorem provides a convergence estimate.
\begin{theorem}\label{Thm Convergence R}
Let $q$ satisfy condition \eqref{Cond on q}.
There exists a constant $C=C(q, x, \ell)$ such that
\begin{equation}\label{R residual estimate}
    \biggl\|\frac{\tilde R(x,t)}{\bigl(1-\frac{t^2}{x^2}\bigr)^{\frac{\ell+1}{2}}} - \left(1-\frac{t^2}{x^2}\right)^{\frac{\ell+1}{2}}\sum_{k=0}^N c_{2k}(\ell+1; h)\tilde P_{2k}^{(\ell+1,\ell+1)}\left(\frac tx\right)\biggr\|_{L_2(-x,x)} \le
    \frac{C}{(2N-\ell-1)^{\ell-\mu+1}}
\end{equation}
for all $2N>\ell+1$.

If additionally $q\in W_1^{2p-1}[0,b]$, the right hand side of the inequality \eqref{R residual estimate} can be improved to
\begin{equation}\label{R residual estimate2}
    \frac{C_1}{(2N-\ell-p-1)^{\ell+p+1}},\qquad 2N>\ell+p+1.
\end{equation}
\end{theorem}

\begin{proof}
Note that
\begin{multline}\label{R residuo and En}
       \biggl\|\frac{\tilde R(x,t)}{\bigl(1-\frac{t^2}{x^2}\bigr)^{\frac{\ell+1}{2}}} - \left(1-\frac{t^2}{x^2}\right)^{\frac{\ell+1}{2}}\sum_{k=0}^N c_{2k}(\ell+1; h)\tilde P_{2k}^{(\ell+1,\ell+1)}\left(\frac tx\right)\biggr\|_{L_2(-x,x)} \\
         = \int_{-x}^x \biggl|\frac{\tilde R(x,t)}{\bigl(1-\frac{t^2}{x^2}\bigr)^{\ell+1}} - \sum_{k=0}^N c_{2k}(\ell+1; h)\tilde P_{2k}^{(\ell+1,\ell+1)}\left(\frac tx\right)\biggr|^2\left(1-\frac{t^2}{x^2}\right)^{\ell+1}\,dt\\
         = x\int_{-1}^1 \biggl| h(z) - \sum_{k=0}^{2N} c_k(\ell+1;h)\tilde P^{(\ell+1,\ell+1)}_k(z)\biggr|^2(1-z^2)^{\ell+1}\,dz = xE_{2N}^2(W_{\ell+1};h).
\end{multline}

Let $m = [[\ell-\mu+1]]$. Due to Proposition \ref{Prop Smoothness R} and Theorem \ref{Thm Approx} we obtain that
\begin{equation}\label{Approx wo omega}
E_{2N}(W_{\ell+1};h)\le \frac{c(\ell+1,m)}{(2N-m)^{m}} \omega\left(W_{\ell+1+m}; h^{(m)};\frac 1{2N-m}\right).
\end{equation}
Let us estimate the modulus of continuity \eqref{Cont modulus}. Due to the Minkowski inequality, it is sufficient to estimate \eqref{Cont modulus} for each term in \eqref{h derivative} separately. We only present the estimates for the first integral. We have for the integrand in \eqref{Cont modulus}
\begin{multline}\label{Modulus integrand}
    \left|\frac{g^{(m-k)}(\cos(\theta+t))p_k(\cos(\theta+t))}{\sin^{2\ell+2+2k}(\theta+t)} - \frac{g^{(m-k)}(\cos \theta)p_k(\cos\theta)}{\sin^{2\ell+2+2k}\theta}\right|^2\sin^{2\ell+3+2m}\theta\\
    \le 2\left|\frac{\bigl(g^{(m-k)}(\cos(\theta+t))-g^{(m-k)}(\cos\theta)\bigr)p_k(\cos(\theta+t))}{\sin^{2\ell+2k+2}(\theta+t)}\right|^2 \sin^{2\ell+3+2m}\theta \\+ 2\left|g^{(m-k)}(\cos\theta)\left(\frac{p_k(\cos(\theta+t))}{\sin^{2\ell+2k+2}(\theta+t)}- \frac{p_k(\cos\theta)}{\sin^{2\ell+2k+2}\theta}\right)\right|^2\sin^{2\ell+2k+3}\theta.
\end{multline}
Consider the first term. If $k=0$, then by the Lipschitz property of the function $g^{(m)}$ (see Proposition \ref{Prop R near pm x}) we have
\begin{multline*}
|g^{(m)}(\cos(\theta+t)) - g^{(m)}(\cos\theta)| \le C|\cos(\theta+t)-\cos\theta|^{\ell-\mu+1-m} \\
= 2C\left|\sin\frac t2\right|^{\ell-\mu+1-m}\cdot \left|\sin\left(\theta+\frac t2\right)\right|^{\ell-\mu+1-m}\le C_1t^{\ell-\mu+1-m}\sin^{\ell-\mu+1-m}\left(\theta+\frac t2\right).
\end{multline*}
Hence the first term can be bounded by
\begin{equation}\label{Estimate for k eq 0}
\frac{C_2 t^{2\ell-2\mu+2-2m} \sin^{2\ell-2\mu+2-2m}(\theta+t/2) \sin^{2\ell+2m+3}\theta}{\sin^{4\ell+4}(\theta+t)}\le C_3 t^{2\ell-2\mu+2-2m} \sin^{1-2\mu}\theta.
\end{equation}
If $k\ge 1$, the function $g^{(m-k)}$ is differentiable, and by the mean value theorem
\[
|g^{(m-k)}(\cos(\theta+t)) - g^{(m-k)}(\cos\theta)| \le (\theta+t-\theta) |g^{(m-k+1)}(\cos\xi)|\cdot |\sin\xi|,
\]
where $\xi\in(\theta,\theta+t)$. For the sake of simplicity we assume that $\theta+t\le \pi/2$, so $\cos\theta>\cos\xi>\cos(\theta+t)$ and $\sin\theta<\sin\xi<\sin(\theta+t)$, the other case is similar. Then due to Proposition \ref{Prop R near pm x}
\begin{align*}
|g^{(m-k+1)}(\cos\xi)|&\le C_1(1-\cos^2\xi)^{\ell-\mu-m+k}\\
&\le (1-\cos^2(\theta+t))^{\ell-\mu-m+k}=\sin^{2\ell-2\mu-2m+2k}(\theta+t)
\end{align*}
(note that $\ell-\mu-m+k\ge 0$ since $m<\ell-\mu+1$). Hence the first term can be bounded by
\begin{multline}\label{Estimate for k gt 0}
\frac{C_2t^2\sin^{4\ell-4\mu-4m+4k}(\theta+t) \sin^2(\theta+t) \sin^{2\ell+2m+3}\theta }{\sin^{4\ell+4k+4}(\theta+t)} =
\frac{C_2t^2\sin^{2\ell+2m+3}\theta }{\sin^{4m+4\mu+2}(\theta+t)} \\
=C_2 t^{2\ell-2\mu+2-2m}\frac{t^{2m+2\mu-2\ell} \sin^{2\ell+2m+3}\theta}{\sin^{4m+4\mu+2}(\theta+t)} \\
\le C_4t^{2\ell-2\mu+2-2m}\frac{\sin^{2m+2\mu-2\ell}t\sin^{2\ell+2m+3}\theta}{\sin^{4m+4\mu+2}(\theta+t)}\le
C_4t^{2\ell+2-2m}\sin^{1-2\mu}\theta,
\end{multline}
where we have used that $m<\ell-\mu+1$ and that $t\le \frac 2\pi \sin t$ for $t\in [0,\pi/2]$.

The proof for the second term in \eqref{Modulus integrand} can be done similarly with the help of the mean value theorem and Proposition \ref{Prop R near pm x}, we left the details to the reader.

As a result, taking into account that $1-2\mu>0$, we obtain that
\begin{equation}\label{Estimate for w}
\omega(W_{\ell+1+m};h^{(m)};\delta)\le C\delta^{\ell-\mu+1-m} = C\delta^{\{\{ \ell-\mu+1\}\}},
\end{equation}
which together with \eqref{Approx wo omega} finishes the proof of \eqref{R residual estimate}.

For the proof of the case when $q\in W_1^{2p-1}[0,b]$, one performs the proof for the function $h_p$ and notes that $h$ and $h_p$ differ by a polynomial of order $p$, hence their best polynomial approximations coincide starting with $n\ge p$.
\end{proof}

As for the pointwise convergence of the series in \eqref{R Jacobi}, we have the following result.
\begin{proposition}\label{Prop R pointwise}
Let $q$ satisfy condition \eqref{Cond on q}.
The series in \eqref{R Jacobi} converges absolutely and uniformly with respect to $t$ on any segment $[-x+\varepsilon,x-\varepsilon]$.
If $q$ is absolutely continuous on $[0,b]$, then the series converges absolutely and uniformly with respect to $t$ on the whole segment $[-x,x]$.

The following weighted estimate holds
\begin{multline}\label{R Jacobi conv pointwise}
    \left|\tilde R(x,t) - \left(1-\frac{t^2}{x^2}\right)^{\ell+1}\sum_{k=0}^n c_{2k}(\ell+1;h) \tilde P_{2k}^{(\ell+1,\ell+1)}\left(\frac tx\right)\right|\le c\left(1-\frac{t^2}{x^2}\right)^{\frac{2\ell+1}4} \frac{\ln n}{n^{\ell+1+p}},\\ t\in [-x,x],\ 2n>\ell+p+1.
\end{multline}
Here the parameter $p$ is from the inclusion $q\in W_1^{2p-1}[0,b]$ (and is taken equal to $-\mu$ if $q\not \in W_1^1[0,b]$).
\end{proposition}

\begin{proof}
The proof immediately follows from Proposition \ref{Prop R near pm x} and \cite[Theorems 7.6 and 7.7]{Suetin}.
\end{proof}

Comparing \eqref{R series1} with \eqref{R Jacobi} we see that
\begin{equation}\label{beta via cFourier}
    \tilde \beta_n(x) = \frac{x c_{2n}(\ell+1;h)}{\sqrt{h_{2n}}},
\end{equation}
where
\[
h_n = \frac{2^{2\ell+3}}{2n+2\ell+3} \frac{\Gamma^2(n+\ell+2)}{n!\Gamma(n+2\ell+3)}
\]
is the square of the norm of the Jacobi polynomial $P^{(\ell+1,\ell+1)}_n$, see \cite[22.2.1]{Abramowitz}.

\begin{corollary}\label{Cor DecayBetas}
Let $q$ satisfy condition \eqref{Cond on q}.
There exists a constant $C=C(q,x,\ell)$ such that
\begin{equation}\label{Estimate for betas}
    \sum_{n=N+1}^\infty \frac{|\tilde\beta_n(x)|^2}{n}\le \frac{C}{N^{2\ell-2\mu+2}},\qquad 2N>\ell+1.
\end{equation}
If additionally $q\in W_1^{2p-1}[0,b]$ then
\begin{equation}\label{Estimate for betas2}
    \sum_{n=N+1}^\infty \frac{|\tilde\beta_n(x)|^2}{n}\le \frac{C}{N^{2\ell+2p+2}},\qquad 2N>\ell+p+1.
\end{equation}
\end{corollary}

\begin{proof}
By \eqref{R residuo and En}, \eqref{FourierJacobiRemainder} and Theorem \ref{Thm Convergence R}
\begin{equation}\label{Estimate for cn}
    \sum_{k=N+1}^\infty |c_{2k}(\ell+1;h)|^2\le \frac{C}{(2N-\ell-1)^{2\ell-2\mu+2}}\le \frac{C_1}{N^{2\ell-2\mu+2}}
\end{equation}
for all $N$ such that $2N>\ell+1$. Since $h_n\ge \frac cn$, see \cite[(IV.7.8)]{Suetin}, \eqref{Estimate for betas} immediately follows from \eqref{beta via cFourier} and \eqref{Estimate for cn}.

Proof of the second statement is similar.
\end{proof}

Leaving only the first term in the sums \eqref{Estimate for betas} and \eqref{Estimate for betas2} we obtain the following result.
\begin{corollary}
Let $q$ satisfy condition \eqref{Cond on q}.
There exists a constant $C_1=C_1(q,x,\ell)$ such that
\begin{equation}\label{Estimate for one beta}
    |\tilde\beta_n(x)|\le \frac{C_1}{n^{\ell-\mu+1/2}},\qquad 2n>\ell+1.
\end{equation}
If additionally $q\in W_1^{2p-1}[0,b]$ then
\begin{equation}\label{Estimate for one beta p}
    |\tilde\beta_n(x)|\le \frac{C_1}{n^{\ell+p+1/2}},\qquad 2n>\ell+p+3/2.
\end{equation}
\end{corollary}

\section{Fourier-Jacobi expansion of the integral kernel $K$}
\label{Sect4}

In this section we apply formula \eqref{KviaR} to the representation \eqref{R Jacobi} and show all the necessary intermediate results like the possibility of termwise differentiation, convergence rate estimates etc.

\begin{theorem}\label{Thm K repr}
Suppose that $q\in AC[0,b]$. Then the representation \eqref{K Jacobi} is valid for the integral kernel $K$.
The series converges absolutely and uniformly with respect to $t\in [0, x-\varepsilon]$ for any small $\varepsilon>0$.
If additionally $q\in W_1^3[0,b]$, then the series converges absolutely and uniformly with respect to $t$ on $[0,x]$.
\end{theorem}

\begin{proof}
Let $m:=[\ell]$, $\lambda :=\{\ell\}$, so $\ell = m+\lambda$.

Suppose initially that $\lambda\ne 0$. Taking $n=m+3$ in the formula \eqref{KviaR} we obtain
\begin{equation}\label{Eq01}
    K(x,t) = \frac{4\sqrt\pi t^{\ell+1}}{\Gamma(\ell+3/2)}\left(-\frac{d}{2tdt}\right)^{m+3}
    \sum_{n=0}^\infty\frac{\tilde \beta_n(x)}{\Gamma(2-\lambda)x^{2\ell+3}}\int_t^x (s^2-t^2)^{1-\lambda} s(x^2-s^2)^{\ell+1} P_{2n}^{(\ell+1,\ell+1)}\left(\frac sx\right)\,ds.
\end{equation}
Here we interchanged the sum and the integral which is possible because the series \eqref{R Jacobi} converges uniformly, see Proposition \ref{Prop R pointwise}, and the factor $s(s^2-t^2)^{1-\lambda}$ is bounded.

To evaluate the integral in \eqref{Eq01} we use the formula (2.21.1.4) from \cite{Prudnikov}. Recall that
\[
P_{2n}^{(\ell+1,\ell+1)}(x) = \frac{(\ell+2)_{2n}}{(2\ell+3)_{2n}} C_{2n}^{\ell+3/2}(x),
\]
where $C_n^\lambda$ are the Gegenbauer polynomials and $(x)_n$ is the Pochhammer symbol, see \cite{Abramowitz}. Then
\begin{equation*}
    \begin{split}
      I_n :=& \int_t^x(s^2-t^2)^{1-\lambda} s(x^2-s^2)^{\ell+1} C_{2n}^{\ell+3/2}\left(\frac sx\right)\,ds \\
        =& \frac{(-1)^n x^{2m+6} (2\ell+3)_{2n} \Gamma(\ell+2)}{2(2n)!} \sum_{k=0}^{m+n+3} \frac{(\lambda-1)_k}{k!} \frac{\Gamma(2-\lambda-k)}{\Gamma(m+n+4-k)} \frac{\Gamma(k+n+\lambda-3/2)}{\Gamma(k+\lambda-3/2)}\left(\frac tx\right)^{2k}.
    \end{split}
\end{equation*}
Here we used that in the formula (2.21.1.4) from \cite{Prudnikov} the sum in the first term is finite and the second term  disappear  due to the presence of the expression $\Gamma(-k)$, $k\in\mathbb{N}_0$ in the denominator. Using the reflection formula $\Gamma(z) \Gamma(1-z) = \frac{\pi}{\sin\pi z}$ we obtain
\[
\frac{(\lambda-1)_k\cdot \Gamma(2-\lambda-k)}{\Gamma(2-\lambda)} = \frac{\Gamma(\lambda+k-1)\Gamma(2-\lambda-k)}{\Gamma(\lambda-1)\Gamma(2-\lambda)} = \frac{\sin\pi(\lambda-1)}{\sin\pi(\lambda +k-1)} = (-1)^k,
\]
hence
\begin{equation}\label{Eq03}
      \frac{I_n}{\Gamma(2-\lambda)}  =\frac{(-1)^n x^{2m+6} (2\ell+3)_{2n} \Gamma(\ell+2)}{2(2n)!} \sum_{k=0}^{m+n+3}  \frac{(-1)^k\Gamma(k+n+\lambda-3/2)}{k!\Gamma(m+n+4-k)\Gamma(k+\lambda-3/2)}\left(\frac tx\right)^{2k}.
\end{equation}
Substituting \eqref{Eq03} in \eqref{Eq01} and noting that $\frac{d}{2tdt}=\frac{d}{dt^2}$ we obtain the expression
\begin{multline}\label{Eq04}
    K(x,t) = \frac{2(-1)^{m+3}\sqrt\pi t^{\ell+1}x^{3-2\lambda}}{\Gamma(\ell+3/2)} \left(\frac{d}{dt^2}\right)^{m+3} \sum_{n=0}^\infty \frac{(-1)^n\tilde\beta_n(x) \Gamma(\ell+2n+2)}{(2n)!}\\
    \times\sum_{k=0}^{m+n+3} \frac{(-1)^k}{k!(m+n+3-k)!}\frac{\Gamma(k+n+\lambda-3/2)}{\Gamma(k+\lambda-3/2)}\left(\frac tx\right)^{2k}.
\end{multline}

To justify the possibility for termwise differentiation of the series in \eqref{Eq04} we show that the series and all termwise derivatives up to the order $m+2$ converge at the point $t=x/\sqrt 2$ and that the termwise derivative of order $m+3$ is uniformly convergent with respect to $t$ on any $[\varepsilon,x-\varepsilon]$ containing the point $x/\sqrt 2$.

We recall that it is possible to introduce the generalized Jacobi polynomials $P^{(\alpha,\beta)}_n$ for arbitrary real parameters $\alpha$ and $\beta$, we refer the reader to \cite[Sect. 4.22]{Szego1959} for details. Note that (see \cite[22.2.2 and 22.5.2]{Abramowitz}) for any $s=0,1,\ldots, m+3$
\begin{align}
    \left(\frac{d}{dt^2}\right)^{s} & \sum_{k=0}^{m+n+3} \frac{(-1)^k}{k!(m+n+3-k)!}\frac{\Gamma(k+n+\lambda-3/2)}{\Gamma(k+\lambda-3/2)}\left(\frac tx\right)^{2k} \nonumber\\
\displaybreak[2]
    =&\sum_{k=s}^{m+n+3}\frac{(-1)^k \Gamma(k+n+\lambda-3/2)}{k!(m+n+3-k)!\Gamma(k+\lambda-3/2)}\frac{k!}{(k-s)! x^{2s}} \left(\frac tx\right)^{2(k-s)}\nonumber\\
\displaybreak[2]
    =&\frac{(-1)^{s}}{(m+n+3-s)!x^{2s}} \sum_{k=0}^{m+n+3-s} (-1)^k\frac{ (m+n+3-s)!}{k! (m+n+3-s-k)!}\frac{\Gamma(k+s+n+\lambda-3/2)}{\Gamma(k+s+\lambda-3/2)} \left(\frac tx\right)^{2k}\nonumber\\
    =&\frac{(-1)^{m+n+3}}{x^{2s}} \frac{\Gamma(2n+\ell+3/2)}{(m+n+3-s)!\Gamma(n+\ell+3/2)}G_{m+n+3-s}\left(2s+\lambda-m-9/2, s+\lambda-\frac 32; \frac {t^2}{x^2}\right) \nonumber\\
    =& \frac{(-1)^{m+n+3}}{x^{2s}} \frac{\Gamma(n+s+\lambda-3/2)}{\Gamma(n+\ell+3/2)}P_{m+n+3-s}^{(s-m-3,s+\lambda-5/2)}\left(2\frac {t^2}{x^2}-1\right),\label{GenJacobi s}
\end{align}
here $G_n(p,q,z)$ are the shifted Jacobi polynomials.

Hence at $t=x/\sqrt 2$ we need to obtain the absolute convergence of the series
\begin{equation}\label{Eq04a}
    \sum_{n=0}^\infty \frac{(-1)^{m+3}\tilde \beta_n(x)}{x^{2s}} \frac{\Gamma(\ell+2n+2)\Gamma(n+s+\lambda -3/2)}{\Gamma(2n+1)\Gamma(n+\ell+3/2)} P_{m+n+3-s}^{(s-m-3,s+\lambda-5/2)}(0).
\end{equation}
Values of the Jacobi polynomials at zero are uniformly bounded by $\frac c{\sqrt n}$, see \cite[Theorem 7.32.2]{Szego1959}. The fraction of Gamma functions behaves as $O(n^{s+\lambda - 2})$ as $n\to\infty$, see \cite[(IV.7.18)]{Suetin}, and since $s\le m+2$, we obtain that $s+\lambda-2\le m+\lambda=\ell$. Due to \eqref{Estimate for one beta p} for $q\in AC[0,b]=W_1^1[0,b]$ we have $|\tilde \beta_n(x)|\le \frac{c}{n^{\ell+3/2}}$. As a result, the terms of the series \eqref{Eq04a} decay at least as $\frac 1{n^2}$ and the series converges absolutely.

For the derivative of order $s=m+3$ we obtain from \eqref{GenJacobi s}
\begin{equation}\label{Eq05}
\begin{split}
\left(\frac{d}{dt^2}\right)^{m+3}&\sum_{k=0}^{m+n+3} \frac{(-1)^k}{k!(m+n+3-k)!}\frac{\Gamma(k+n+\lambda-3/2)}{\Gamma(k+\lambda-3/2)}\left(\frac tx\right)^{2k} \\
&= \frac{(-1)^{m+n+3}}{x^{2m+6}} P_n^{(0,\ell+1/2)}\left(2\frac {t^2}{x^2}-1\right) = \frac{(-1)^{m+1}}{x^{2m+6}} P_n^{(\ell+1/2,0)}\left(1-2\frac {t^2}{x^2}\right).
\end{split}
\end{equation}
The Jacobi polynomials in \eqref{Eq05} are uniformly bounded by $\frac{c}{\sqrt n}$ for  $t\in[\varepsilon,x-\varepsilon]$, see \cite[Theorem 7.5]{Suetin}. Due to \cite[(IV.7.18)]{Suetin} we have $\frac{\Gamma(\ell+2n+2)}{\Gamma(2n+1)}=O(n^{\ell+1})$, $n\to\infty$. So the uniform convergence of the series
\[
\sum_{n=0}^\infty \frac{(-1)^{n+m+1} \tilde\beta_n(x)\Gamma(\ell+2n+2)}{(2n)!}
    P_n^{(\ell+1/2,0)}\left(1-\frac{2t^2}{x^2}\right)
\]
would follow from the absolute convergence of the series
\begin{equation}\label{M estimate}
    \sum_{n=1}^\infty \frac{|\tilde \beta_n(x)|}{\sqrt n} n^{\ell+1}.
\end{equation}
Consider
\begin{align}
      \displaybreak[2]
      \sum_{n=2^m+1}^{2^{m+1}} &\frac{|\tilde \beta_n(x)|}{\sqrt n}n^{\ell+1}
      \le (2^{m+1})^{\ell+1} \sum_{n=2^m+1}^{2^{m+1}} \frac{|\tilde \beta_n(x)|}{\sqrt n} \le (2^{m+1})^{\ell+1} \sqrt{2^m} \left( \sum_{n=2^m+1}^{2^{m+1}} \frac{|\tilde \beta_n(x)|^2}{n}\right)^{1/2}\nonumber\\
        & \le (2^{m+1})^{\ell+1} \sqrt{2^m} \left( \sum_{n=2^m+1}^{\infty} \frac{|\tilde \beta_n(x)|^2}{n}\right)^{1/2}\le 2^{(m+1)(\ell+1)+m/2} \frac{1}{(2^{m})^{\ell+2}} = \frac{2^{\ell+1}}{2^{m/2}},\label{Eq07}
\end{align}
where we have used the estimate \eqref{Estimate for betas2} with $p=1$. Hence
\[
    \sum_{n=1}^\infty \frac{|\tilde \beta_n(x)|}{\sqrt n} n^{\ell+1} \le |\tilde \beta_1(x)|+\sum_{m=0}^\infty \frac{2^{\ell+1}}{2^{m/2}} < \infty,
\]
which proves the uniform convergence of the termwise derivative of order $m+3$.

Substituting \eqref{Eq05} into \eqref{Eq04} we obtain the representation \eqref{K Jacobi}.

In the presented proof the factor $t^{\ell+1}$ was left outside the sum, this was required to show termwise differentiability of the series in \eqref{Eq04}. Including this factor $t^{\ell+1}$ we can show the uniform convergence of \eqref{K Jacobi} for $t\in [0,x-\varepsilon]$. Indeed, as it follows from \cite[Theorem 7.5]{Suetin}, there is a constant $c$ such that for all $n$ and $t\in [0,x]$
\[
\left(\frac{t^2}{x^2}\right)^{\frac{\ell+1}{2}} \left(1-\frac{t^2}{x^2}\right)^{\frac 14} \left|P_n^{(\ell+1/2,0)}\left(1-\frac{2t^2}{x^2}\right)\right| \le \frac{c}{\sqrt n}.
\]
Hence for any $\varepsilon>0$ there is a constant $c_\varepsilon$ such that for all $n\ge 0$
\[
t^{\ell+1}\left|P_n^{(\ell+1/2,0)}\left(1-\frac{2t^2}{x^2}\right)\right| \le \frac{c_\varepsilon}{\sqrt n},\qquad t\in [0,x-\varepsilon],
\]
and the same proof involving the absolute convergence of the series \eqref{M estimate} shows uniform convergence of \eqref{K Jacobi} for $t\in [0,x-\varepsilon]$.

The uniform convergence on the whole segment $[0,x]$ can be obtained similarly taking into account that $P_n^{(\ell+1/2,0)}(1) = 1$ imposing additional requirements on the decay rate of the coefficients $\tilde \beta_n$, and, hence, on the smoothness of the potential to be able to apply the estimate \eqref{Estimate for betas2}. We left the details to the reader.

Now suppose that $\lambda=0$, i.e., $\ell\in\mathbb{Z}$. Taking $n=\ell+2$ in the formula \eqref{KviaR} we obtain
\begin{equation*}
\begin{split}
    K(x,t) &= \frac{4\sqrt\pi t^{\ell+1}}{\Gamma(\ell+3/2)}\left(-\frac{d}{dt^2}\right)^{\ell+1} \frac 1{2t} \frac d{dt}\int_x^t s\tilde R(x,s)\,ds \\
    & = \frac{2\sqrt \pi (-t)^{\ell+1}}{\Gamma(\ell+3/2)} \left(\frac{d}{dt^2}\right)^{\ell+1} \sum_{n=0}^\infty \frac{\tilde \beta_n(x)}{x} \left(1-\frac{t^2}{x^2}\right)^{\ell+1} P_{2n}^{(\ell+1,\ell+1)}\left(\frac tx\right).
\end{split}
\end{equation*}
Using the formulas (22.5.20) and (22.5.21) from \cite{Abramowitz} we obtain
\[
P_{2n}^{(\ell+1,\ell+1)}\left(\frac tx\right) = \frac{(\ell+2)_{2n} (\ell+3/2)_{n}}{(2\ell+3)_{2n} (1/2)_n} P_n^{(\ell+1,-1/2)}\left(2\frac{t^2}{x^2}-1\right).
\]
While formula (4.22.2) from \cite{Szego1959} states
\[
\left(1-\frac{t^2}{x^2}\right)^{\ell+1} P_n^{(\ell+1,-1/2)}\left(2\frac{t^2}{x^2}-1\right) = (-1)^{\ell+1}\frac{\Gamma(n+1/2)(n+\ell+1)!}{\Gamma(n+\ell+3/2) n!} P_{n+\ell+1}^{(-\ell-1,-1/2)}\left(2\frac{t^2}{x^2}-1\right).
\]
Hence
\begin{multline}\label{Weighted P via Neg P}
    \left(1-\frac{t^2}{x^2}\right)^{\ell+1} P_{2n}^{(\ell+1,\ell+1)}\left(\frac tx\right)\\ = (-1)^{\ell+1} \frac{(\ell+2)_{2n} (\ell+3/2)_{n}}{(2\ell+3)_{2n} (1/2)_n}\frac{\Gamma(n+1/2)(n+\ell+1)!}{\Gamma(n+\ell+3/2) n!} P_{n+\ell+1}^{(-\ell-1,-1/2)}\left(2\frac{t^2}{x^2}-1\right).
\end{multline}

Using the Legendre duplication formula $\Gamma(z)\Gamma(z+1/2) = 2^{1-2z}\sqrt\pi \Gamma(2z)$ we obtain
\begin{gather*}
    (\ell+3/2)_n (n+\ell+1)! = \frac{\Gamma(n+\ell+3/2)}{\Gamma(n+\ell+2)}{\Gamma(\ell+3/2)} = \frac{\sqrt\pi \Gamma(2n+2\ell+3)}{2^{2n+2\ell+2}\Gamma(\ell+3/2)},\\
    (1/2)_n n! = \frac{\Gamma(n+1/2)\Gamma(n+1)}{\Gamma(1/2)} = \frac{\Gamma(2n+1)}{2^{2n}},\qquad
    \Gamma(\ell+2)\Gamma(\ell+3/2) = \frac{\sqrt\pi \Gamma(2\ell+3)}{2^{2\ell+2}}.
\end{gather*}
Using these identities it follows from \eqref{Weighted P via Neg P} that
\begin{equation}\label{Weighted P via Neg P 2}
    \left(1-\frac{t^2}{x^2}\right)^{\ell+1} P_{2n}^{(\ell+1,\ell+1)}\left(\frac tx\right) = (-1)^{\ell+1}\frac{\Gamma(2n+\ell+2)\Gamma(n+1/2)}{(2n)!\Gamma(n+\ell+3/2)}
    P_{n+\ell+1}^{(-\ell-1,-1/2)}\left(2\frac{t^2}{x^2}-1\right).
\end{equation}
Observe that if one takes \eqref{Eq04} and applies 2 derivatives (and utilizes formula \eqref{Eq05} for $s=2$), one obtains a series containing exactly expressions \eqref{Weighted P via Neg P 2}. That is, the proof for the case of integer $\ell$ can be finished exactly the same as it was done in the case on non-integer $\ell$.
\end{proof}

\begin{lemma}\label{Lemma Orthonormal Base}
Let $x>0$ be fixed. The system of functions
\begin{equation}\label{Orthonormal Base}
\left\{ \frac{\sqrt{4n+2\ell+3}}{x^{\ell+3/2}} \cdot t^{\ell+1}P_n^{(\ell+1/2,0)} \left(1-\frac{2t^2}{x^2}\right) \right\}_{n=0}^\infty
\end{equation}
forms an orthonormal basis of $L_2(0,x)$.
\end{lemma}
\begin{proof}
The orthonormality of the functions from \eqref{Orthonormal Base} follows from the following equality
\begin{multline}\label{Scalar product}
\int_0^x t^{2\ell+2} P_n^{(\ell+1/2,0)}\left(1-\frac{2t^2}{x^2}\right) P_m^{(\ell+1/2,0)}\left(1-\frac{2t^2}{x^2}\right)\,dt \\
= \frac{x^{2\ell+3}}{2^{\ell+5/2}} \int_{-1}^1 (1-z)^{\ell+1/2} P_n^{(\ell+1/2,0)}(z)P_m^{(\ell+1/2,0)}(z)\,dz = \frac{x^{2\ell+3}}{2^{\ell+5/2}} \delta_{nm} g_m,
\end{multline}
where $\delta_{nm}$ is the Kronecker delta symbol and $g_m=\frac{2^{\ell+3/2}}{2m+\ell+3/2}$ is the square of the norm of the Jacobi polynomial $P_m^{(\ell+1/2,0)}$.

Observe that the linear span of the first $N+1$ functions from \eqref{Orthonormal Base} coincides with the linear span of the system $\{ t^{2n+\ell+1}\}_{n=0}^N$. Hence the completeness of the system \eqref{Orthonormal Base} immediately follows from the M\"{u}ntz theorem, \cite[Chapter 11, \S5]{DevoreLorentz}.
\end{proof}

\begin{theorem}\label{Thm Convergence K}
Let $q$ satisfy condition \eqref{Cond on q} and $x>0$ be fixed. Then the series in \eqref{K Jacobi} converges in $L_2(0,x)$ to $K(x,t)$.

Let $q\in W_1^{2p-1}[0,b]$ for some $p\in\mathbb{N}$ and $x>0$ be fixed. Then for the truncated series
\begin{equation}\label{KN Jacobi}
\begin{split}
    K_N(x,t):=&\sum_{n=0}^N \frac{2(-1)^n \sqrt\pi\tilde\beta_n(x)\Gamma(\ell+2n+2)}{(2n)!\Gamma(\ell+3/2)x^{2\ell+3}}t^{\ell+1}
    P_n^{(\ell+1/2,0)}\left(1-\frac{2t^2}{x^2}\right) \\
    = &\frac{t^{\ell+1}}{x^{\ell+2}} \sum_{n=0}^N \beta_n(x) P_n^{(\ell+1/2,0)}\left(1-\frac{2t^2}{x^2}\right)
\end{split}
\end{equation}
the following estimate holds
\begin{equation}\label{K Jacobi Estimate}
    \left\|K(x,t) - K_N(x,t)\right\|_{L_2(0,x)}\le \frac{C}{N^{p}}
\end{equation}
for all $N$ satisfying $2N>\ell+p+1$.
\end{theorem}

\begin{proof}
Suppose initially that $q\in AC[0,b]$, that is, Theorem \ref{Thm K repr} holds.

Note that by Lemma \ref{Lemma Orthonormal Base} the terms of the series \eqref{K Jacobi} are orthogonal.
Hence, up to a multiplicative constant (dependent on $x$), the $L_2(0,x)$ norm of the series in \eqref{K Jacobi} equals to \[
\sum_{n=0}^\infty \frac{|\tilde \beta_n(x)|^2}{4n+2\ell+3} \frac{\Gamma^2(2n+\ell+2)}{\Gamma^2(2n+1)} <c\sum_{n=0}^\infty |\tilde \beta_n(x)|^2 \cdot n^{2\ell+1}.
\]
where we used that $\frac{\Gamma(2n+\ell+2)}{\Gamma(2n+1)} = O(n^{\ell+1})$, $n\to\infty$. Now we proceed similarly to \eqref{Eq07}.
\begin{equation}\label{Eq08}
\sum_{n=2^m}^{2^{m+1}-1} |\tilde \beta_n(x)|^2 \cdot n^{2\ell+1} \le (2^{m+1})^{2\ell+2} \sum_{n=2^m}^{2^{m+1}-1} \frac{|\tilde \beta_n(x)|^2 }{n} \le \frac{2^{(m+1)(2\ell+2)}}{(2^m)^{2\ell+4}} = \frac{2^{2\ell+2}}{2^{2m}},
\end{equation}
where we used \eqref{Estimate for betas2} for $p=1$. The last inequality proves $L_2$ convergence of the series in \eqref{K Jacobi}.

Since the series in \eqref{K Jacobi} converges in $L_2(0,x)$ and converges pointwise to $K(x,t)$ for all $t\in(0,x)$, the $L_2$ limit of the series is equal to $K(x,t)$. Due to the orthogonality of the terms of the series
\begin{equation*}
    \left\|K(x,t) - K_N(x,t)\right\|^2_{L_2(0,x)}
    =\frac{4\pi}{\Gamma^2(\ell+3/2) x^{2\ell+3}} \sum_{n=N+1}^\infty  \frac{|\tilde \beta_n(x)|^2}{4n+2\ell+3} \frac{\Gamma^2(2n+\ell+2)}{\Gamma^2(2n+1)}\le \frac{C}{N^{2p}},
\end{equation*}
the proof of the last inequality is similar to \eqref{Eq08} with the use of \eqref{Estimate for betas2} with the parameter $p$.

Let us suppose now that $q$ satisfies \eqref{Cond on q} only. Since the integral kernel $K$ as a function of $t$ belongs to $L_2(0,x)$, it can be expanded into Fourier series with respect to system \eqref{Orthonormal Base}. We can write this expansion in the following form
\begin{equation}\label{K fourier}
    K(x,t) = \sum_{n=0}^\infty \frac{\alpha_n(x)}{x^{\ell+2}} \cdot t^{\ell+1}P_n^{(\ell+1/2,0)} \left(1-\frac{2t^2}{x^2}\right),
\end{equation}
where
\begin{equation}\label{Eq betan via K}
    \alpha_n(x) = \frac{(4n+2\ell+3) }{x^{\ell+1}} \int_0^x K(x,t)\cdot t^{\ell+1}P_n^{(\ell+1/2,0)}\left(1-2\frac{t^2}{x^2}\right)\,dt.
\end{equation}
Note that the coefficients $\beta_n$ are defined by \eqref{R coeff1}, \eqref{beta via cFourier} and \eqref{beta and tilde beta}. For absolutely continuous potentials coefficients $\alpha_n$ and $\beta_n$ coincide as follows from the first part of this theorem. Considering a sequence of absolutely continuous potentials $q_n$ converging to the potential $q$  in $L_1((0,b), \nu(x)dx)$ (see notations of Appendix \ref{AppB}) and using Corollary \ref{Corr continuity K} we can conclude that $\alpha_n=\beta_n$ for all $n\in\mathbb{N}_0$ for any potential $q$ satisfying condition \eqref{Cond on q}.
\end{proof}

%

\begin{remark}
The condition $q\in AC[0,b]$ may be excessive for Theorem \ref{Thm K repr} to hold. In Subsection \ref{Subsect 83} we numerically illustrate that the coefficients $\beta_n$ decay much faster than in the estimate \eqref{Estimate for betas2}, so even the condition \eqref{Cond on q} may turn to be sufficient for validity not only of Theorem \ref{Thm Convergence K}, but also of Theorem \ref{Thm K repr}.
\end{remark}

\section{Representation of the regular solution}
\label{Sect5}
In this section we substitute the representation \eqref{K Jacobi} for the integral kernel $K$ into \eqref{VolkTransmute} and obtain the representation \eqref{u NSBF} for the regular solution.

First we need the following lemma.
\begin{lemma}\label{Lemma Bessel j}
\begin{equation}\label{Eq Bessel J}
    \int_0^x t^{\ell+3/2} P_n^{(\ell+1/2,0)} \left(1-2\frac{t^2}{x^2}\right) J_{\ell+1/2}(\omega t)\,dt = \frac{x^{\ell+3/2}}{\omega} J_{\ell+2n+3/2}(\omega x).
\end{equation}
\end{lemma}

\begin{proof}
Applying the change of variable $z=2\frac{t^2}{x^2}-1$ the integral converts to
\[
I_n:=\frac{(-1)^n x^{\ell+5/2}}{2^{\ell/2+9/4}}\int_{-1}^1 (z+1)^{\ell/2+1/4}P_n^{(0,\ell+1/2)}(z) J_{\ell+1/2}\left(\frac{\omega x}{\sqrt 2} \sqrt{z+1}\right)\,dz.
\]
We will calculate this integral using the formula 2.22.12.3 from \cite{Prudnikov}. We would like to point out that direct application of this formula results in an expression containing
\[
(0)_n\cdot \,_2F_3\left(\ell+\frac 32, 1; \ell+n+\frac 52, 1-n, \ell+\frac 32; -\frac{\omega^2x^2}4\right),
\]
that is, $0$ in the numerator and $0$ in the denominator due to the negative integer parameter in the hypergeometric function. To overcome this difficulty we apply the classical technique of passage to the limit.
Consider for $\varepsilon>0$ the integral
\[
I^\varepsilon_n:=\frac{(-1)^n x^{\ell+5/2}}{2^{\ell/2+9/4}}\int_{-1}^1 (z+1)^{\ell/2+1/4}P_n^{(0,\ell+1/2+\varepsilon)}(z) J_{\ell+1/2}\left(\frac{\omega x}{\sqrt 2} \sqrt{z+1}\right)\,dz.
\]
It is easy to see that $I_n^\varepsilon\to I_n$ as $\varepsilon\to 0$. On the other hand, the formula 2.22.12.3 from \cite{Prudnikov} gives
\begin{align*}
I_n^\varepsilon &= \frac{(-1)^nx^{\ell+5/2}}{2^{\ell/2+9/4}}\frac{(-1)^n (\varepsilon)_n}{n!\Gamma(\ell+3/2)} \cdot 2B\left(\ell+\frac 32,n+1\right) \left(\frac{\omega x}{\sqrt 2}\right) ^{\ell+1/2}\\
 &\quad \times\,_2F_3\left(\ell+\frac 32,1-\varepsilon;\ell+n+\frac 52,1-n-\varepsilon,\ell+\frac 32; -\frac{\omega^2x^2}4\right)\\
 &=-\frac{x^{\ell+5/2}(\varepsilon+1)_{n-1}B(\ell+3/2,n+1)}{2^{\ell/2+5/4}n!\Gamma(\ell+3/2)}\left(\frac{\omega x}{\sqrt 2}\right) ^{\ell+1/2} \\
 &\quad\times \bigl((1-n-\varepsilon) + (n-1)\bigl)\cdot \,_1F_2\left(1-\varepsilon;\ell+n+\frac 52,1-n-\varepsilon; -\frac{\omega^2x^2}4\right),
\end{align*}
where $_2F_3$ was reduced to $_1F_2$ due to two equal parameters. Applying in the last expression the formula
(see \cite[(4.21.5)]{Szego1959})
\begin{equation*}\label{passage to limit}
\lim_{c\to 1-m}(c+m-1)\cdot\,_1F_2(a;b,c;x) \\= (-1)^{m-1}\frac{(a)_m x^m}{m!(m-1)! (b)_m} \,_1F_2(a+m; b+m, m+1; x)
\end{equation*}
and noting also that the hypergeometric function $_1F_2$ reduces to $_0F_1$ due to two pairs of equal parameters we obtain that
\begin{align*}
    I_n&= -\frac{x^{\ell+5/2}}{2^{\ell/2+5/4}} \frac{(n-1)!B\left(\ell+3/2,n+1\right)}{n! \Gamma(\ell+3/2)}  \left(\frac{\omega x}{\sqrt 2}\right)^{\ell+1/2} \\
    \displaybreak[2]
    &\quad\times \frac{(-1)^{n-1} n!}{n! (n-1)! (\ell+n+5/2)_n} \left(-\frac{\omega^2 x^2}{4}\right)^n \,_0F_1\left( ;\ell+2n+\frac 52; -\frac{\omega^2x^2}{4}\right) \\
    \displaybreak[2]
    &= \frac{x^{\ell+5/2} (\omega x)^{2n+\ell+1/2}}{ 2^{2n+\ell+3/2} \Gamma(\ell+2n+5/2)}\,_0F_1\left( ;\ell+2n+\frac 52; -\frac{\omega^2x^2}{4}\right)\\
    &= \frac{x^{\ell+5/2} (\omega x)^{2n+\ell+1/2}}{ 2^{2n+\ell+3/2} \Gamma(\ell+2n+5/2)} \frac{\Gamma(\ell+2n+5/2) 2^{\ell+2n+3/2}}{(\omega x )^{\ell+2n+3/2}} J_{\ell+2n+3/2}(\omega x) \\
    & =\frac{x^{\ell+3/2}}\omega J_{\ell+2n+3/2}(\omega x),
\end{align*}
where formula 9.1.69 from \cite{Abramowitz} was used.
\end{proof}

\begin{theorem}\label{thm Rep Sol}
Let $q$ satisfy condition \eqref{Cond on q}. Then the regular solution $u(\omega,x)$ of \eqref{PertBessel} satisfying the asymptotic condition \eqref{bl asympt} as $x\to 0$ admits the following representation
\begin{equation}\label{Sol Repr}
    u(\omega, x) = \omega x j_\ell(\omega x) + \sum_{n=0}^\infty \beta_n(x) j_{\ell+2n+1}(\omega x),
\end{equation}
where the coefficients $\beta_n$ are related to $\tilde \beta_n$ by \eqref{beta and tilde beta}.
The series \eqref{Sol Repr} converges absolutely and uniformly with respect to $\omega$ on any compact subset of the complex plane.

Suppose additionally that $q\in W_1^{2p-1}[0,b]$ for some $p\in\mathbb{N}$. Denote by
\begin{equation}\label{uN}
u_N(\omega, x) = \omega x j_\ell(\omega x) + \sum_{n=0}^N \beta_n(x) j_{\ell+2n+1}(\omega x)
\end{equation}
an approximate solution obtained by truncating the series in \eqref{Sol Repr}. Then the following estimate holds uniformly for $\omega\in\mathbb{R}$
\begin{equation}\label{Estimate for Sol}
    |u(\omega, x) - u_N(\omega, x)| \le \frac{c(x)}{N^p},\qquad 2N>\ell+p+1,
\end{equation}
here $p$ is the parameter of smoothness of the potential $q$.
\end{theorem}

\begin{proof}
Substituting \eqref{K Jacobi} into \eqref{Transmut}, changing the integral and the sum (which is possible due to $L_2$-convergence of \eqref{K Jacobi}) and applying Lemma \ref{Lemma Bessel j} we obtain
\begin{align*}
    u(\omega, x) &= \omega x j_\ell(\omega x) + \sum_{n=0}^\infty \frac{2\sqrt\pi (-1)^n \tilde\beta_n(x)\Gamma(\ell+2n+2)}{x^{2\ell+3} \Gamma(\ell+3/2) (2n)!}  \int_0^x t^{\ell+1} P_n^{(\ell+1/2,0)} \left(1-2\frac{t^2}{x^2}\right) \omega t j_\ell(\omega t)\,dt \\
    &=\omega x j_\ell(\omega x) + \sum_{n=0}^\infty \frac{2\sqrt\pi (-1)^n \tilde\beta_n(x)\Gamma(\ell+2n+2)}{x^{2\ell+3} \Gamma(\ell+3/2) (2n)!}  \\
    &\quad \times \sqrt{\frac{\pi \omega} 2}\int_0^x t^{\ell+1} P_n^{(\ell+1/2,0)} \left(1-2\frac{t^2}{x^2}\right) \sqrt{ t} J_{\ell+1/2}(\omega t)\,dt \\
    &= \omega x j_\ell(\omega x) + \sum_{n=0}^\infty \frac{2\sqrt{\pi} (-1)^{n} \tilde \beta_n(x) \Gamma(\ell+2n+2)}{x^{\ell+1}\Gamma(\ell+3/2)(2n)!} j_{\ell+2n+1}(\omega x).
\end{align*}

For the proof of convergence rate estimate note that
\[
u_N(\omega, x) = \omega x j_\ell(\omega x) + \int_0^x K_N(x,t) \omega t j_\ell(\omega x)\,dt,
\]
where $K_N$ was introduced in \eqref{KN Jacobi}. Hence recalling that the function $b_\ell(z) = zj_\ell(z)$ is bounded on the whole real line, see \cite[(9.2.1)]{Abramowitz}, and is bounded on the compact subsets of $\mathbb{C}$, and applying the Cauchy-Schwarz inequality we obtain
\begin{align*}
\displaybreak[2]
|u(\omega, x) - u_N(\omega,x)| &= \left|\int_0^x \bigl(K(x,t) - K_N(x,t)\bigr) b_\ell(\omega t)\,dt\right|\\
&\le \|K(x,\cdot)-K_N(x,\cdot)\|_{L_2(0,x)} \left(\int_0^x |b_\ell(\omega t)|^2\,dt\right)^{1/2} \le \frac{c(x)}{N^p}.
\qedhere
\end{align*}
\end{proof}

\section{Representation for the derivative of the regular solution}
\label{Sect6}
In this section we obtain a representation for the derivative of the regular solution. We are looking for a representation possessing the same remarkable properties of uniform with respect to $\omega\in\mathbb{R}$ error bounds and exponentially fast convergence for smooth potentials. For that, instead of differentiating the representation \eqref{Sol Repr} with respect to $x$, we differentiate \eqref{Transmut} and proceed similarly to what were done in Sections \ref{Sect3}--\ref{Sect5}. First we recall some facts about the derivative, the transmutation operator and Mehler's representation from \cite{KTC2017}.

Let $u(\omega,x)$ be the regular solution of \eqref{PertBessel} satisfying the asymptotics \eqref{bl asympt} as $x\to 0$. Then
\begin{equation}\label{u prime via transmut}
    u'(\omega, x) = \omega b_\ell'(\omega x) + \frac{Q(x)}2 b_\ell(\omega x) + \int_0^x K_1(x,t) b_\ell(\omega t)\,dt,
\end{equation}
here $K_1(x,t)$ denotes $\partial_x K(x,t)$.

For $q\in AC[0,b]$ the following Mehler's type representation holds
\begin{equation*}
    d(\omega)u'(\omega, x) = d(\omega) \left(\omega b_\ell'(\omega x) + \frac{Q(x)}2 b_\ell(\omega x)\right)+\int_{-x}^x \tilde R^{(2)}(x,t) e^{i\omega t}\,dt,
\end{equation*}
where $\tilde R^{(2)}$ is an even, compactly supported on $[-x,x]$ function satisfying $\tilde R^{(2)}(x,\cdot) \in W_2^{\ell+3/2-\varepsilon}(\mathbb{R})$.

The integral kernels $K_1$ and $\tilde R^{(2)}$ are related by the following formula
\begin{equation}\label{R2 via K1}
    \tilde R^{(2)}(x,s) = \frac{\Gamma \left( \ell+\frac{3}{2}\right) }{\sqrt{\pi }\Gamma (\ell+1)}
\int_{s}^{x}K_1(x,t)t^{-\ell}(t^{2}-s^{2})^{\ell}\,dt.
\end{equation}

The following estimates hold. If $q\in AC[0,b]$, then
\begin{equation*}
    \left| u'(\omega,x ) -\omega b_\ell'(\omega x) - \frac{Q(x)}2 b_\ell(\omega x) \right|\le \frac{C}{|\omega|},\qquad \omega \in \mathbb{R}.
\end{equation*}
And if $q\in W_1^{2p+1}[0,b]$ for some $p\in \mathbb{N}$, then
\begin{equation*}
    u'(\omega,x ) =\omega b_\ell'(\omega x) - \frac{Q(x)}2 b_\ell(\omega x) + \sum_{k=1}^p \tilde A_k(x) \frac{J_{\ell+k+1/2}(\omega x)}{\omega^{k-1}} + \mathcal{R}_{p+1,x}(\omega, x),
\end{equation*}
where
\begin{equation*}
    |\mathcal{R}_{p+1,x}(\omega ,x)|\le \frac {c}{|\omega|^{p+1}},\qquad \omega\in\mathbb{R}.
\end{equation*}
By the Paley-Wiener theorem
\[
d(\omega)\mathcal{R}_{p+1,x}(\omega, x) = \int_{-x}^x \tilde R^{(2)}_p(x,t) e^{i\omega t}\,dt,
\]
where $\tilde R^{(2)}_p$ as a function of $t$ is even, compactly supported on $[-x,x]$ and satisfies $\tilde R^{(2)}_p\in W_2^{\ell+3/2+p-\varepsilon}(\mathbb{R})$. Moreover,
\[
\tilde R^{(2)}(x,t) - \tilde R^{(2)}_p(x,t) = \left(1-\frac{t^2}{x^2}\right)^{\ell+1} P_p(x,t),
\]
where $P_p$ is a polynomial in $t$ of degree less or equal to $p$.

Comparing the above formulas for the derivative with the corresponding ones for the solution, one can see that the same relations hold between the integral kernels $K$ and $\tilde R$ and between the integral kernels $K_1$ and $\tilde R^{(2)}$. Similar estimates for the solution and its derivative hold, similar improvement exists depending on the smoothness of the potential $q$. The only difference is that the requirement $q\in W_1^{2p-1}[0,b]$ should be changed to $q\in W_1^{2p+1}[0,b]$. With this change, the scheme presented in Sections \ref{Sect3}--\ref{Sect5} can be applied without significant changes for the derivative. We left the details for the reader and only present the final results.

\begin{proposition}
Let $q\in W_1^{2p+1}[0,b]$ for some $p\in \mathbb{N}_0$. Then the integral kernel $\tilde R^{(2)}$ has the following representation
\begin{equation}\label{R2 Jacobi}
    \tilde R^{(2)}(x,t) = \left(1-\frac{t^2}{x^2}\right)^{\ell+1}\sum_{n=0}^\infty \frac{\tilde\gamma_n(x)}{x} P_{2n}^{(\ell+1,\ell+1)}\left(\frac tx\right).
\end{equation}
Denote by
\[
\tilde R^{(2)}_N(x,t) = \left(1-\frac{t^2}{x^2}\right)^{\ell+1}\sum_{n=0}^N \frac{\tilde\gamma_n(x)}{x} P_{2n}^{(\ell+1,\ell+1)}\left(\frac tx\right).
\]
Let $x>0$ be fixed. Then the following estimates hold
\begin{equation}\label{Estimate 1 for R2}
    \biggl\| \frac{\tilde R^{(2)}(x,t) - \tilde R_N^{(2)}(x,t)}{(1-t^2/x^2)^{(\ell+1)/2}}\biggr\|_{L_2(-x,x)}\le \frac{c_1}{(2N-\ell-p-1)^{\ell+p+1}},\quad 2N>\ell+p+1,
\end{equation}
and
\begin{equation}\label{Estimate 2 for R2}
\bigl|\tilde R^{(2)}(x,t) - \tilde R_N^{(2)}(x,t)\bigr| \le c_2\left(1-\frac{t^2}{x^2}\right)^{\frac{2\ell+1}{4}} \frac{\ln N}{N^{\ell+1+p}}, \qquad t\in [-x,x],\ 2N>\ell+p+1.
\end{equation}
\end{proposition}

Comparing \eqref{R2 via K1} with \eqref{RviaK} we may conclude from \eqref{KviaR}  that
\begin{equation}
K_1(x,t)=\frac{4\sqrt{\pi }}{\Gamma \left( \ell+\frac{3}{2}\right) }\frac{t^{\ell+1}}{\Gamma (n-\ell-1)} \left( -\frac{d}{2tdt}\right) ^{n}\int_{t}^{x}(s^{2}-t^{2})^{n-\ell-2}s\tilde R^{(2)}(x,s)ds,  \label{K1 via R2}
\end{equation}
here $n$ can be arbitrary integer satisfying $n>\ell+1$. And similarly to Section \ref{Sect4} we obtain the following result.

\begin{theorem}\label{Thm K1 repr}
Suppose $q\in W_1^{2p+1}[0,b]$ for some $p\in\mathbb{N}$. Then the following representation
\begin{equation}\label{K1 Jacobi}
\begin{split}
    K_1(x,t) &= \frac{2\sqrt\pi }{x^{2\ell+3}\Gamma(\ell+3/2)}\sum_{n=0}^\infty \frac{(-1)^n \tilde\gamma_n(x)\Gamma(\ell+2n+2)}{(2n)!}
    t^{\ell+1}P_n^{(\ell+1/2,0)}\left(1-\frac{2t^2}{x^2}\right)\\
    &=\sum_{n=0}^\infty  \frac{\gamma_n(x)}{x^{\ell+2}}t^{\ell+1}P_n^{(\ell+1/2,0)}\left(1-\frac{2t^2}{x^2}\right),
\end{split}
\end{equation}
is valid. Here we defined
\begin{equation}\label{gamma and tilde gamma}
    \gamma_n(x) = (-1)^{n} \frac{2\sqrt{\pi} \tilde \gamma_n(x)\Gamma(\ell+2n+2)}{x^{\ell+1}\Gamma(\ell+3/2)(2n)!}.
\end{equation}
The series converges absolutely and uniformly with respect to $t\in [0, x-\varepsilon]$ for any small $\varepsilon>0$ and for any fixed $x>0$ converges in $L_2(-x,x)$.

For the truncated series
\begin{equation}\label{K1N Jacobi}
    K_{1,N}(x,t):=\sum_{n=0}^N \frac{\gamma_n(x)}{x^{\ell+2}}t^{\ell+1}
    P_n^{(\ell+1/2,0)}\left(1-\frac{2t^2}{x^2}\right)
\end{equation}
the following estimate holds
\begin{equation}\label{K1 Jacobi Estimate}
    \left\|K_1(x,t) - K_{1,N}(x,t)\right\|_{L_2(0,x)}\le \frac{C}{N^{p}}
\end{equation}
for all $N$ satisfying $2N>\ell+p+1$.
\end{theorem}

And finally substituting \eqref{K1 Jacobi} into \eqref{u prime via transmut} we obtain the following result.
\begin{theorem}\label{thm Rep DerSol}
Let $q\in W_1^{2p+1}[0,b]$. Then the derivative of the regular solution $u(\omega,x)$ considered in Theorem \ref{thm Rep Sol} has the following representation
\begin{equation}\label{DerSol Repr}
    u'(\omega, x) = \omega^2 x j_{\ell-1}(\omega x) + \left(\frac{xQ(x)}2 - \ell\right)\omega j_\ell(\omega x) +
     \sum_{n=0}^\infty \gamma_n(x) j_{\ell+2n+1}(\omega x),
\end{equation}
where the coefficients $\gamma_n$ are related to the coefficients $\tilde \gamma_n$ by \eqref{gamma and tilde gamma}.
The series \eqref{DerSol Repr} converges absolutely and uniformly with respect to $\omega$ on any compact subset of the complex plane.

Denote by
\begin{equation}\label{uN prime}
u'_N(\omega, x) = \omega^2 x j_{\ell-1}(\omega x) + \left(\frac{xQ(x)}2 - \ell\right)\omega j_\ell(\omega x) +
     \sum_{n=0}^N \gamma_n(x) j_{\ell+2n+1}(\omega x),
\end{equation}
an approximate solution obtained by truncating the series in \eqref{DerSol Repr}. Then the following estimate holds uniformly for $\omega\in\mathbb{R}$
\begin{equation}\label{Estimate for DerSol}
    |u'(\omega, x) - u'_N(\omega, x)| \le \frac{c(x)}{N^p},\qquad 2N>\ell+p+2.
\end{equation}
\end{theorem}

\begin{remark}
The validity of representations \eqref{K1 Jacobi} and \eqref{DerSol Repr} can be established under a weaker condition on the potential $q$, at least under the condition $q\in AC[0,b]$. The proof can be done similarly to Sections \ref{Sect4}, \ref{Sect5} and Appendix \ref{AppB}. See also proof of Theorem 7.1 from \cite{KTC2017} for estimations of derivatives of regular solutions. We left the details for a separate work.
\end{remark}

\section{Recurrent formulas for the coefficients $\beta_n$ and $\gamma_n$}
\label{Sect7}
Note that the functions $\beta_n$ are twice differentiable and grow not faster than polynomially in $n$ (which can be seen from \eqref{R coeff1}, \eqref{beta via cFourier} and \eqref{beta and tilde beta}, see also \cite{KTC2017}). The same is true for $\beta_n'$ and $\beta_n''$. Taking into account the inequality $|j_n(x)|\le \sqrt\pi \bigl|\frac x2\bigr|^n \frac{1}{\Gamma(n+3/2)}$, $x\in\mathbb{R}$, one may verify that the series \eqref{Sol Repr} can be differentiated twice termwise.

So let us substitute \eqref{Sol Repr} into equation \eqref{PertBessel}. Similarly to \cite[Section 6]{KTC2017} we obtain that
\begin{equation}\label{Substitute 1}
    \begin{split}
      0 & = L[u(\omega, x)]-\omega^2 u(\omega,x) = q(x) b_\ell(\omega x) +  \\
        &\quad  +\sum_{n=0}^{\infty}\left[  \beta_{n}(x)\left(
j_{2n+\ell+1}(\omega x)\left(  q(x)-\frac{2n(2n+2\ell+1)}{x^2}\right)
-\frac{2\omega}{x}j_{2n+\ell+2}(\omega x)\right)  \right. \\
&\quad \left.   -\beta_{n}^{\prime\prime}(x)j_{2n+\ell+1}(\omega x)+2\beta_{n}^{\prime
}(x)\left(  \omega j_{2n+\ell+2}(\omega x)-\frac{2n+\ell+1}{x}j_{2n+\ell+1}(\omega x)\right)
\right]  .
    \end{split}
\end{equation}
After applying the formula
\begin{equation}
j_{2n+\ell+1}(\omega x)=\frac{\omega x}{4n+2\ell+3}\bigl(j_{2n+\ell}(\omega x)+j_{2n+\ell+2}(\omega
x)\bigr) \label{recusrive Bessel}%
\end{equation}
equality \eqref{Substitute 1} can be rewritten in the form
\begin{equation*}
    \sum_{n=0}^\infty \alpha_n(x) j_{2n+\ell}(\omega x) = 0,
\end{equation*}
where
\begin{equation}\label{sum alpha j}
\alpha_0(x) = q(x) + \frac{1}{2\ell+3}\left(-\beta_0''(x) - \frac{2\ell+2}{x}\beta_0'(x) + \beta_0(x)q(x)\right)
\end{equation}
and
\begin{align*}
\alpha_{n}(x)    :=&  \frac{1}{4n+2\ell+3}\left(
-\beta_{n}^{\prime\prime}(x)-\frac{2(2n+\ell+1)}{x}\beta_{n}^{\prime}(x)+\left(
q(x) - \frac{2n(2n+\ell+1)}{x^2}\right)  \beta_{n}(x)\right)
 \\
\displaybreak[2]
&  +\frac{1}{4n+2\ell-1}\left(  -\beta_{n-1}^{\prime\prime}(x)-\frac{2(
2n+\ell-1)  }{x}\beta_{n-1}^{\prime}(x)\right. \\
& \left.+\left(q(x) -  \frac{2(n-1)
(2n+2\ell-1)}{x^2}\right)  \beta
_{n-1}(x)\right)  +2\left(  \frac{\beta_{n-1}^{\prime}(x)}{x}-\frac{\beta_{n-1}%
(x)}{x^{2}}\right).
\end{align*}
Due to the orthogonality of the functions $j_{2n+\ell}(\omega x)$ on $(0,\infty)$ we obtain from \eqref{sum alpha j} that $\alpha_n(x)\equiv 0$, $n\in \mathbb{N}_0$, which can be written in the following form (c.f.\ \cite{KNT 2015} and \cite{KTC2017})
\begin{align*}
    \frac{1}{(2\ell+3)x^{\ell+1}}L\left[x^{\ell+1}\beta_0(x)\right] &= -q(x),\\
    \frac{1}{(4n+2\ell+3)x^{2n+\ell+1}}L\left[x^{2n+\ell+1}\beta_n(x)\right] &= -\frac{x^{2n+\ell}}{4n+2\ell-1}L\left[\frac{\beta_{n-1}(x)}{x^{2n+\ell}}\right],\qquad n\ge 1.
\end{align*}

Note that due to \eqref{R coeff1}, \eqref{beta via cFourier} each of the coefficients $\tilde \beta_k$ is a linear combination of the Fourier-Legendre coefficients of the integral kernel $\tilde R$ studied in \cite{KTC2017}. Hence, using (4.12) from \cite{KTC2017} and  \eqref{beta and tilde beta} one can see that the functions $\beta_n$ satisfy the following initial conditions
\[
\beta_n(0) = 0, \qquad n\in\mathbb{N}_0.
\]

Let $u_0$ denotes the regular solution of the equation $Lu=0$ normalized by the asymptotic condition
\[
u_0(x) \sim x^{\ell+1},\qquad x\to 0,
\]
and suppose that $u_0(x)\ne 0$ for all $x\in (0,b]$. If the regular solution of the equation $Lu=0$ vanish in $(0,b]$ it is always possible to perform a spectral shift such that new equation possesses a non-vanishing solution, see Appendix \ref{AppA}.
Then a solution $u$ of an equation
\begin{equation}\label{eq Lu h}
    Lu=h
\end{equation}
can be easily obtained using the P\'{o}lya factorization of $L$, $Lu = -\frac 1{u_0} \frac{d}{dx} u_0^2 \frac{d}{dx} \frac u{u_0}$. The function
\begin{equation}\label{u Polya}
    u(x) = -u_0(x) \int_0^x \frac 1{u_0^2(t)} \int_0^t u_0(s) h(s)\, ds\,dt
\end{equation}
is a solution of \eqref{eq Lu h}. Note also that the expression \eqref{u Polya} gives the unique solution of \eqref{eq Lu h} satisfying $u(x) = o(x^{\ell+1})$, $x\to 0$. And since the function $x^{2n+\ell+1} \beta_n(x) = o(x^{\ell+1})$, $x\to 0$, the expression \eqref{u Polya} can be used to reconstruct the functions $\beta_n$, $n\in\mathbb{N}$.

It is easy to verify that the function $\beta_0$ is given by the expression
\begin{equation}\label{beta0}
    \beta_0(x) = (2\ell+3) \left(\frac{u_0(x)}{x^{\ell+1}}-1\right).
\end{equation}

Hence one can start with \eqref{beta0} and define for $n\ge 1$
\begin{equation*}
    \beta_n(x) = \frac{4n+2\ell+3}{4n+2\ell-1}\frac{u_0(x)}{x^{2n+\ell+1}}\int_0^x \frac 1{u_0^2(t)} \int_0^t u_0(s) s^{4n+2\ell}L\left(\frac{\beta_{n-1}(s)}{s^{2n+\ell}}\right)\, ds\,dt.
\end{equation*}
To eliminate the first and second derivatives of $\beta_{n-1}$ resulting from the term $L\left(\frac{\beta_{n-1}(s)}{s^{2n-1}}\right)$, one may apply the integration by parts and obtain (similarly to \cite{KNT 2015}) the following recurrent formulas.
\begin{align}
    \eta_n(x) &= \int_0^x \bigl(tu_0'(t)+(2n+\ell)u_0(t)\bigr)t^{2n+\ell-1}\beta_{n-1}(t)\,dt,\label{etan}\\
    \theta_n(x) &= \int_0^x \frac{1}{u_0^2(t)}\bigl( \eta_n(t) - t^{2n+\ell} \beta_{n-1}(t) u_0(t)\bigr)\,dt,\qquad n\ge 1,\label{thetan}
\end{align}
and finally
\begin{equation}\label{beta_n alt}
    \beta_n(x) = -\frac{4n+2\ell+3}{4n+2\ell-1}\left[\beta_{n-1}(x) + \frac{2(4n+2\ell+1)u_0(x)\theta_n(x)}{x^{2n+\ell+1}}\right].
\end{equation}

To obtain the formulas for the coefficients $\gamma_n$ we proceed as follows. Differentiating \eqref{Sol Repr} we have that
\[
u^{\prime}(\omega,x)=\omega b_{\ell}^{\prime}(\omega x)+\sum_{n=0}^{\infty
}\left(  \beta_{n}^{\prime}(x)j_{2n+\ell+1}(\omega
x)-\omega\beta_{n}(x)j_{2n+\ell+2}(\omega x)+\tfrac{2n+\ell+1}{x}\beta_{n}(x)j_{2n+\ell+1}(\omega
x)\right)  .
\]
Comparing this expression with \eqref{DerSol Repr}, utilizing \eqref{recusrive Bessel} and rearranging terms, we arrive at the equality
\begin{equation*}
    \left(\frac{Q(x)}{2} +\frac{1}{2\ell+3} \left(\gamma_0(x) - \beta_0'(x) - \frac{\ell+1}{x}\beta_0(x)\right)\right)j_\ell(\omega x) +\sum_{n=1}^\infty
    \tilde \alpha_n(x) j_{2n+\ell}(\omega x) = 0,
\end{equation*}
where
\[
\begin{split}
\tilde \alpha_n(x) &:= \frac{1}{4n+2\ell+3}\left(\gamma_n(x) -\beta_n'(x)-\frac{2n+\ell+1}{x}\beta_n(x)\right) \\
&\quad + \frac{1}{4n+2\ell-1}\left(\gamma_{n-1}(x) - \beta_{n-1}'(x) + \frac{2n+\ell}{x}\beta_{n-1}(x)\right).
\end{split}
\]
Again, due to the orthogonality of the functions $j_{\ell+2n}(\omega x)$ on $(0,\infty)$ we obtain that the coefficients $\tilde \alpha_n$ are equal to 0, that is,
\begin{align*}
    \gamma_0(x) &= \beta_0'(x) + \frac{\ell+1}{x}\beta_0(x) -\frac{2\ell+3}2 Q(x),\\
    \gamma_n(x) &=\beta_n'(x)+\frac{2n+\ell+1}{x}\beta_n(x) -\frac{4n+2\ell+3}{4n+2\ell-1}\left(\gamma_{n-1}(x) - \beta_{n-1}'(x) + \frac{2n+\ell}{x}\beta_{n-1}(x)\right).
\end{align*}
Using \eqref{beta0} and \eqref{beta_n alt} these formulas can be written as
\begin{align}\label{gamma0}
    \gamma_0(x) &= (2\ell+3)\left(\frac{u_0'(x)}{x^{\ell+1}} - \frac{\ell+1}x -\frac{Q(x)}{2}\right),\\
    \gamma_n(x) &=-\frac{4n+2\ell+3}{4n+2\ell-1}\left[\gamma_{n-1}(x) + (4n+2\ell+1)\left( \frac{2u_0'(x)\theta_n(x)}{x^{2n+\ell+1}} + \frac{2\eta_n(x)}{u_0(x) x^{2n+\ell+1}} - \frac{\beta_{n-1}(x)}{x}\right)\right].\label{gamma_n alt}
\end{align}

\section{Numerical examples}
\label{Sect8}
The general scheme of application of the proposed representation \eqref{Sol Repr} and \eqref{DerSol Repr} for the solution of boundary value and spectral problems for equation \eqref{PertBessel} is basically the same as for the paper \cite{KTC2017}. One starts with computing a particular solution $u_0$ of the equation $Lu_0=0$, then computes the coefficients $\{\beta_n\}_{n=0}^{N_1}$ and $\{\gamma_n\}_{n=0}^{N_2}$ and obtains approximations for the regular solution and its derivative. The particular solution together with its derivative can be calculated using the SPPS representation \cite[Section 3]{CKT2013}. The assumption for the solution $u_0$ to be non-vanishing automatically holds if $q(x)\ge 0$, $x\in (0,b]$. For other cases one may need to apply the spectral shift as described in Appendix \ref{AppA}. The coefficients $\beta_n$ and $\gamma_n$ can be computed using \eqref{etan}--\eqref{beta_n alt} and \eqref{gamma_n alt}. To perform an indefinite integration numerically one may use a piecewise polynomial interpolation or approximation and integrate the resultant polynomial. For machine precision arithmetics we use the fifth degree polynomial interpolation (resulting in the 6 point Newton-Cottes rule), in previous papers we also used spline interpolation (much slower) or Clenshaw-Curtis rule (mainly for higher precision arithmetics). We refer the reader to \cite{KNT 2015} and \cite{KTC2017} for further details.

We would like to mention that the coefficients $\beta_n(x)$ and $\gamma_n(x)$ decay as $n\to\infty$, see \eqref{Estimate for one beta p} and \eqref{beta and tilde beta}, and decay faster then polynomially for smooth potentials $q$. However computation errors propagate through the recurrent formulas \eqref{beta_n alt} and \eqref{gamma_n alt}, meaning that after several dozens of coefficients the computation error will be comparable with the value of the coefficient itself. Equality \eqref{K Goursat} can be used to estimate an optimum number of the coefficients $\beta_n$. One can write \eqref{K Goursat} using the representation \eqref{K Jacobi} and \eqref{beta and tilde beta} as
\begin{equation}\label{Verification beta}
    \frac{xQ(x)}{2} = \sum_{n=0}^\infty (-1)^{n}\beta_n(x)
\end{equation}
and take as an optimal value for $N_1$ the index minimizing the discrepancy between the left hand side and the truncated sum. In order to present similar equality allowing one to estimate an optimal number of the coefficients $\gamma_n$, we conjecture that the the following formula holds.
\begin{equation}\label{K1 Goursat}
    K_1(x,x) = \frac{Q(x)}8 + \frac{q(x)}4 - \frac{(2\ell+1)}4 q(0).
\end{equation}
For the potentials possessing holomorphic extension onto the disk of radius $2xe\sqrt{1+2|\ell|}$ one can easily verify that \eqref{K1 Goursat} holds using formulas (2.3) and (4.6) from \cite{Chebli1994}. For the general case approximation by smooth potentials (e.g., polynomials) should work, however we do not want to enter into details in this paper.

On the base of \eqref{K1 Goursat} and \eqref{K1 Jacobi} we obtain the following equality
\begin{equation}\label{Verification gamma}
x\left(\frac{Q(x)}8 + \frac{q(x)}4 - \frac{(2\ell+1)}4 q(0)\right) = \sum_{n=0}^\infty (-1)^n\gamma_n(x),
\end{equation}
allowing one to estimate an optimal number $N_2$ of coefficients $\gamma_n$ to use.

All the computations were performed in machine precision using Matlab 2017. All the functions involved were represented by their values on an uniform mesh of 5001 points. The analytic expression was used only to obtain the value of the potential at the mesh points, all other computations were done numerically. We would like to emphasize that our aim was to show that even straightforward implementation of the proposed representations can deliver in almost no time results which are comparable or even superior to those produced by the best existing software packages.

\subsection{Example 1}
\label{Ex SP1}
Consider the following spectral problem
\begin{gather}
-u''+\left(\frac{\ell(\ell+1)}{x^2}+x^2\right) u=\omega^2 u, \quad 0\le x\le
\pi,\label{Ex1Eq}\\
u(\omega, \pi)=0.\label{Ex1BC}
\end{gather}
The regular solution of equation \eqref{Ex1Eq} can be written as
\[
u(\omega, x) = x^{\ell+1}e^{x^2/2}\,_1F_1\left(\frac{\omega^2+2\ell+3}{4}; \ell+\frac
32; -x^2\right)
\]
allowing one to compute with any precision arbitrary sets of eigenvalues
using, e.g., Wolfram Mathematica.

The value $\ell=3/2$ was considered in \cite[Example 2]{BoumenirChanane}, \cite[Example 7.3]{CKT2013} and \cite[Example 9.1]{KTC2017}. We compared the results with those obtained using \eqref{uN} with $N=12$. Exact eigenvalues together with the absolute errors of the approximate eigenvalues obtained using different methods are presented in Table \ref{Ex1Table1}. The proposed method is abbreviated as New NSBF. NSBF corresponds to the original representation from \cite{KTC2017} used with $N=100$.
As one can see from the results, the proposed method easily outperforms both the SPPS method and the previous NSBF method. The relative error of all found eigenvalues was less than $2.5\cdot 10^{-15}$, making it comparable with the best available software packages like \textsc{Matslise} \cite{LedouxVDaele}. It is worth to mention that the computation time for the proposed method was less than 0.25 sec on Intel i7-7600U equipped notebook.

\begin{table}[htb!]
\centering
\begin{tabular}{ccccccc}\hline
$n$ & $\omega_{n}$ (Exact/\textsc{Matslise})  & $\Delta \omega_n$ (New NSBF) & $\Delta \omega_n$ (NSBF) & $\Delta \omega_n$ (SPPS) &  $\Delta \omega_n$ (\cite{BoumenirChanane})\\\hline
1 &  $2.46294997397397$ & $4.4\cdot 10^{-16}$ & $1.4\cdot 10^{-14}$ & $2.7\cdot 10^{-13}$ &  $9.4\cdot 10^{-7}$\\
2 &  $3.28835292994256$ & $4.4\cdot 10^{-16}$ & $5.2\cdot 10^{-14}$ & $6.7\cdot 10^{-12}$ &  $1.4\cdot 10^{-5}$\\
3 &  $4.14986421874478$ & $1.8\cdot 10^{-15}$ & $1.2\cdot 10^{-13}$ & $8.2\cdot 10^{-13}$ &  $3.1\cdot 10^{-5}$\\
5 &  $6.00758145811600$ & $8.9\cdot 10^{-16}$ & $6.6\cdot 10^{-13}$ & $5.0\cdot 10^{-13}$ &  $4.1\cdot 10^{-6}$\\
7 &  $7.93973737689930$ & $2.0\cdot 10^{-14}$ & $2.9\cdot 10^{-13}$ & $6.0\cdot 10^{-13}$ & \\
10 & $10.8861250916173$ & $1.8\cdot 10^{-14}$ & $1.5\cdot 10^{-12}$ & $8.6\cdot 10^{-13}$ & \\
20 & $20.8202301908124$ & $3.6\cdot 10^{-15}$ & $1.4\cdot 10^{-11}$  & $9.6\cdot 10^{-14}$ &\\
50 & $50.7786768095149$ & $5.0\cdot 10^{-14}$ & $8.7\cdot 10^{-11}$ & $1.3\cdot 10^{-10}$ & \\
100 &$100.764442245651$ & $4.3\cdot 10^{-14}$ & $9.4\cdot 10^{-9}$ & $5.3\cdot 10^{-4}$ &\\
\hline
\end{tabular}
\caption{The eigenvalues for the spectral problem \eqref{Ex1Eq}, \eqref{Ex1BC} for $\ell=3/2$ compared to the results reported in \cite{CKT2013} and \cite{KTC2017}. Since eigenvalues produced by the \textsc{Matslise} package coincide with the exact eigenvalues to all reported digits, we present them in the combined column. $\Delta\omega_n$ denotes the absolute error of the computed eigenvalue $\omega_n$.}
\label{Ex1Table1}
\end{table}

We also compared the results provided by the
proposed method to those of \cite{KShT2018} and \cite{KTC2017} where other
methods based on the transmutation operators and Neumann series of Bessel functions were implemented. The following
values of $\ell$ were considered: $-0.5$, $0.5$, $1$, $5$, $10$ and $10.5$. We
present the results on Figure \ref{Ex1Fig1}. One can appreciate that the proposed method produces eigenvalues with non-deteriorating error and outperforms the other two methods. Also, for non-integer values of the parameter $\ell$, only few coefficients are necessary in comparison with hundreds required for the method from \cite{KTC2017}.

\begin{figure}[htb!]
\centering
\begin{tabular}{cc}
$\ell=1$ & $\ell=-0.5$ \\
\includegraphics[bb=0 0 216 144, width=3in,height=2in]
{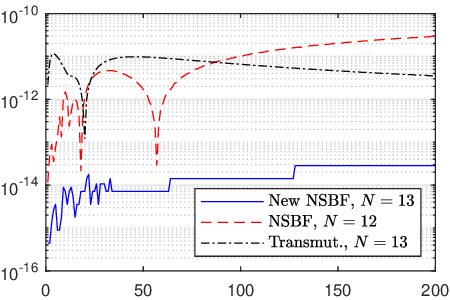} &
\includegraphics[bb=0 0 216 144, width=3in,height=2in]
{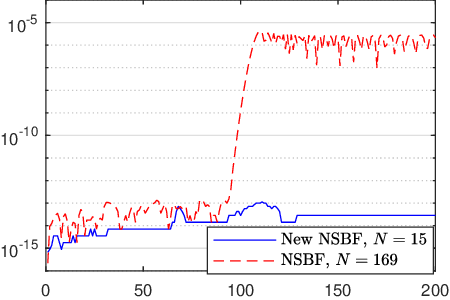} \\
$\ell=5$ & $\ell=0.5$ \\
\includegraphics[bb=0 0 216 144, width=3in,height=2in]
{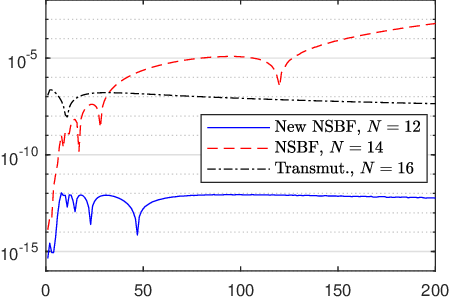} &
\includegraphics[bb=0 0 216 144, width=3in,height=2in]
{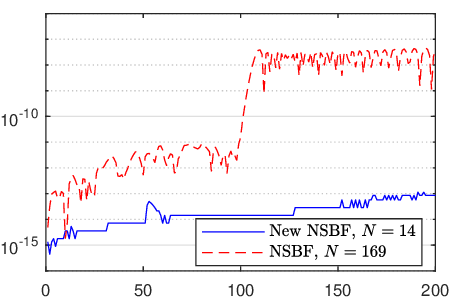} \\
$\ell=10$ & $\ell=10.5$ \\
\includegraphics[bb=0 0 216 144, width=3in,height=2in]
{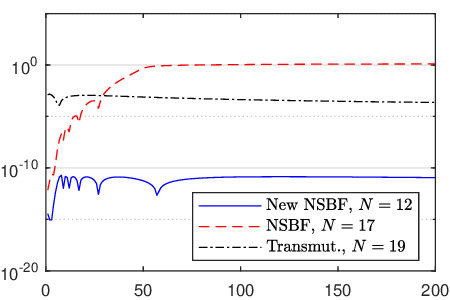} &
\includegraphics[bb=0 0 216 144, width=3in,height=2in]
{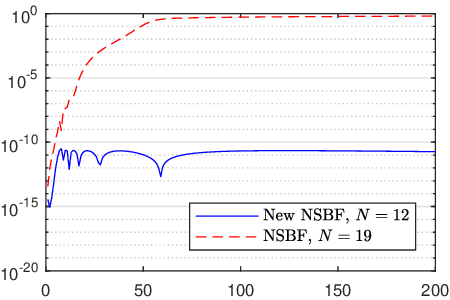}
\end{tabular}
\caption{Absolute errors of the first 200 eigenvalues for the spectral problem \eqref{Ex1Eq}, \eqref{Ex1BC} for different values of $\ell$. On the legends `New NSBF' refers to the proposed method, `NSBF' to the method from \cite{KTC2017} and `Transmut.' to the method from \cite{KShT2018}. The number $N$ indicates the truncation parameter used for calculation of the approximate solution.}
\label{Ex1Fig1}
\end{figure}

\subsection{Example 2}
\label{Ex SP2}
Consider the same equation as in Subsection \ref{Ex SP1} with a different boundary condition:
\begin{equation}\label{Ex2BC}
u'(\omega, \pi)=0.
\end{equation}

For this problem we compare the proposed method with the original NSBF representation \cite[Example 7.2]{KTC2017}. We considered three values of the parameter $\ell$, $-1/2$, $1/2$ and $1$ and computed 200 approximate eigenvalues.
Absolute errors of the obtained eigenvalues are presented on Figure \ref{Ex2Fig1}, on the left -- the proposed method, on the right -- the original NSBF representation. One can see the advantage of the proposed method, especially for non-integer values of $\ell$.

\begin{figure}[htb!]
\centering
\includegraphics[bb=0 0 216 173, width=3in,height=2.4in]
{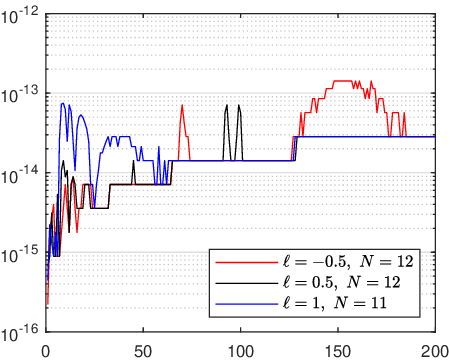}\ \ \
\includegraphics[bb=0 0 216 173, width=3in,height=2.4in]
{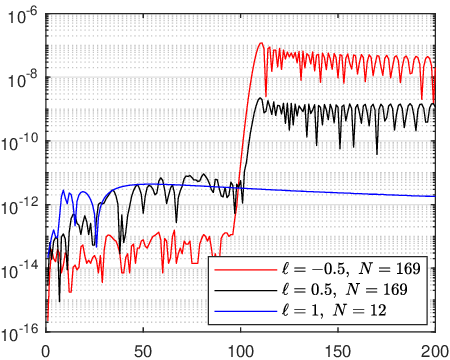}
\caption{Absolute errors of the first 200 eigenvalues for the spectral problem \eqref{Ex1Eq}, \eqref{Ex2BC} for different values of $\ell$. On the left: the proposed representation. On the right: the original NSBF representation from \cite{KTC2017}.
On the legends the truncation parameter $N$ used for calculation of the approximate solution is shown next to the value of the parameter $\ell$.}
\label{Ex2Fig1}
\end{figure}

\subsection{Decay of the coefficients $\beta_n$}
\label{Subsect 83}
In this example we would like to illustrate numerically how far or close are the estimates of Theorem \ref{Thm Convergence R} and Corollary \ref{Cor DecayBetas} from the computed decay of the coefficients $\beta_n$. Taking into account \eqref{beta and tilde beta}, estimate \eqref{Estimate for betas2} states that for $q\in W_1^{2p-1}[0,b]$ one has
\begin{equation}\label{Estim for beta}
\sum_{n=N+1}^\infty \frac{|\beta_n(x)|^2}{n^{2\ell+3}} \le \frac{C_1(x)}{N^{2\ell+2p+2}}.
\end{equation}
Assuming the coefficients $\beta_n$ are nicely behaved, one may suppose they should satisfy
\begin{equation}\label{Assume for beta}
    |\beta_n(x)|\le \frac{C_2(x)}{n^p},\qquad n\ge\frac{\ell+p}2+1.
\end{equation}
Note that it is more like a guess and not a proven inequality. The proven inequality \eqref{Estimate for one beta p} is worse by a factor of $\sqrt n$.

We considered the following potentials
\begin{equation}\label{potentials qm}
q_m(x) = \begin{cases} 1, & x\le \pi/2,\\
1+(x-\pi/2)^m, & x>\pi/2,\ m=0\ldots 5.
\end{cases}
\end{equation}
It is easy to see that $q_m\in W_1^m[0,\pi]$, but $q_m\not\in W_1^{m+1}[0,\pi]$. We took $\ell=1$ and for each of the potentials we computed the coefficients $\beta_k$, $k\le 100$. On the left in Figure \ref{Ex3Fig1} we present a log-log plot of the values $|\beta_k(\pi)|$ against $k$. Such type of plot allows one to reveal dependencies of the form $|\beta_k(\pi)|\sim c\cdot k^\alpha$. The exponent $\alpha$ corresponds to the slope of the linear part of the curve. We obtained the following values for the exponent $\alpha$ (the values are reported up to 2 decimal digits).
\begin{center}
\begin{tabular}{|c|c|c|c|c|c|c|}
\hline
$m$ & 0 & 1 & 2 & 3 & 4 & 5  \\
\hline
$\alpha$ & 1.45 & 3.38 & 3.39 & 5.30 & 5.34 & 7.31 \\
\hline
\end{tabular}
\end{center}
As one can see, the exponent indeed increases when the smoothness of the potential increases by 2, but the exponent also increases in the steps of 2 (approximately), not in the steps of 1 as predicted by \eqref{Assume for beta}. And even the decay rate for the discontinuous potential $q_0$ is better that those that the prediction \eqref{Assume for beta} gives for absolutely continuous potential, suggesting Theorem \ref{Thm K repr} may hold under weaker assumptions on $q$, \eqref{Cond on q} may be enough.

\begin{figure}[htb!]
\centering
\includegraphics[bb=0 0 216 173, width=3in,height=2.4in]
{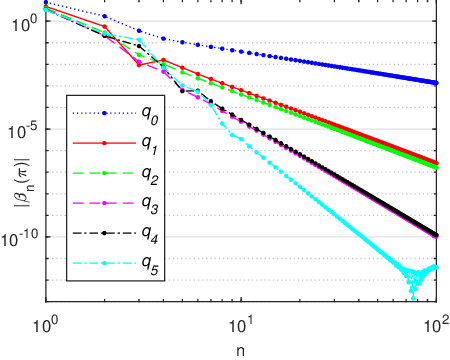}\ \ \
\includegraphics[bb=0 0 216 173, width=3in,height=2.4in]
{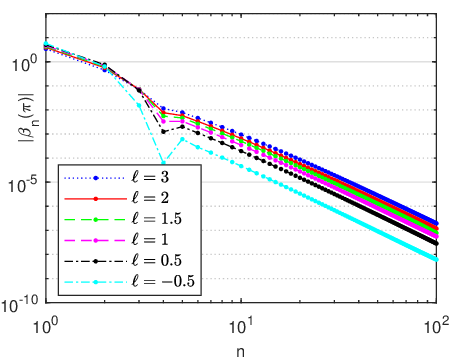}
\caption{Log-log plots of the values $|\beta_k(\pi)|$ against $k$. On the left: for different potentials $q_m$ given by \eqref{potentials qm}. On the right: for the potential $q(x) = \sqrt x$ but for different values of $\ell$.}
\label{Ex3Fig1}
\end{figure}

To verify the dependence of the decay of the coefficients $\beta_n$ on $\ell$, we considered the potential $q(x) = \sqrt x$ and computed the coefficients $\beta_k$, $k\le 100$ for different values of $\ell$. On the right in Figure \ref{Ex3Fig1} we present a log-log plot of the values $|\beta_k(\pi)|$ against $k$. As one can see, the exponents $\alpha$ in the dependencies $|\beta_k(\pi)|\sim c\cdot k^\alpha$ are essentially the same and do not depend on $\ell$.

\appendix
\section{Continuous dependence of regular solutions and integral kernels on potentials}
\label{AppB}
Many proofs related to transmutation operators are performed for smooth potentials and are followed by a phrase ``considering an approximation by smooth potentials and passing to the limit'' to show the validity for not so smooth potentials. However we are not aware of a rigorous result showing continuous dependence of the integral kernel on the potential for perturbed Bessel equations with potentials satisfying condition \eqref{Cond on q}. For that reason we decided to include the proof in the present paper.

Consider two equations
\begin{equation}\label{Eq Bessel ContDep}
    -y_i'' + \frac{\ell(\ell+1)}{x^2} y_i + q_i(x) y_i = \omega^2 y_i,\qquad i=1,2,\ x\in (0,b],
\end{equation}
where potentials $q_1$ and $q_2$ are complex valued functions satisfying condition \eqref{Cond on q}. Without loss of generality we may assume that the parameter $\mu$ is equal for both potentials.

Let $y_i(\omega, x)$, $i=1,2$ denote regular solutions of \eqref{Eq Bessel ContDep} satisfying the asymptotics \eqref{bl asympt}. We are going to estimate $\| y_1(\cdot, x) - y_2(\cdot, x)\|_{L_2(\mathbb{R})}$ for each fixed $x>0$.
Following the notations from \cite{KosSakhTesh2010} and according to \cite[Lemma 2.2]{KosSakhTesh2010}, solutions $y_{1,2}$ satisfy the integral equations
\[
y_i(\omega, x) = b_\ell(\omega x) + \int_0^x G_\ell(\omega, x, t) q_i(t) y_i(\omega, t)\,dt,
\]
where $G_\ell(\omega, x, t)$ is Green's function of the initial value problem. Hence their difference satisfies the equation
\begin{align}
\displaybreak[2]
    y_1(\omega, x) - y_2(\omega, x) &= \int_0^x G_\ell(\omega, x, t) \bigl(q_1(t) y_1(\omega, t) - q_2(t) y_2(\omega, t)\bigr)\,dt\nonumber\\
    \displaybreak[2]
\displaybreak[2]
    & = \int_0^x G_\ell(\omega, x, t) \bigl(q_1(t) - q_2(t)\bigr) y_1(\omega, t)\,dt \nonumber\\
    &\quad + \int_0^x G_\ell(\omega, x, t)q_2(t) \bigl( y_1(\omega, t) - y_2(\omega, t)\bigr)\,dt.\label{Eq for diff}
\end{align}
The following estimates hold \cite[(A.1), (A.2) and (2.18)]{KosSakhTesh2010} for all $\omega \in\mathbb{R}$
\begin{align*}
    |G_\ell(\omega, x, t)| &\le \begin{cases}
    C\left(\frac{x}{b+|\omega| x}\right)^{\ell+1} \left(\frac{b+|\omega| t}{t}\right)^\ell ,& \ell>-1/2,\\
    C\left(\frac{xt}{(b+|\omega| x)(b+|\omega| t)}\right)^{1/2} \left(1 - \log \frac tb\right) ,& \ell=-1/2,
    \end{cases}\\
    |y_1(\omega, x)| & \le C_{q_1}\left(\frac{|\omega|x}{b+|\omega| x}\right)^{\ell+1}\biggl(1+ \int_0^x \frac{s\tilde q_1(s)}{b+|\omega| s}\,ds\biggr) \le \tilde C_{q_1}\left(\frac{|\omega|x}{b+|\omega| x}\right)^{\ell+1},
\end{align*}
where $\tilde q_1$ is defined in \eqref{tilde q} and the constant $\tilde C_{q_1}$ can be bounded by
\[
\tilde C_{q_1} \le C\exp\biggl( \frac{C}{b} \int_0^x t\tilde q_1(t)\,dt\biggr).
\]
Suppose initially that $\ell>-1/2$. We obtain from \eqref{Eq for diff} that
\begin{equation}\label{Diseq for diff}
\begin{split}
    |y_1(\omega, x) - y_2(\omega, x)| &\le C\tilde C_{q_1} \left(\frac{|\omega|x}{b+|\omega| x}\right)^{\ell+1}\int_0^x \frac{s|q_1(s) - q_2(s)|}{b+|\omega|s}\,ds\\
    &\quad + C\left(\frac{x}{b+|\omega| x}\right)^{\ell+1} \int_0^x \left(\frac{b+|\omega| t}{t}\right)^\ell |q_2(t)|\cdot
    |y_1(\omega, t) - y_2(\omega, t)|\,dt.
\end{split}
\end{equation}
Let us divide \eqref{Diseq for diff} by $\bigl(\frac{x}{b+|\omega| x}\bigr)^{\ell+1}$ to obtain
\begin{equation}\label{Diseq for diff2}
\begin{split}
    \left(\frac{b+|\omega| x}{x}\right)^{\ell+1}&|y_1(\omega, x) - y_2(\omega, x)| \le C\tilde C_{q_1}|\omega|^{\ell+1}  \int_0^x \frac{s|q_1(s) - q_2(s)|}{b+|\omega|s}\,ds\\
    &\quad + C\int_0^x \frac{t|q_2(t)|}{b+|\omega| t}\cdot \left(\frac{b+|\omega| t}{t}\right)^{\ell+1}
    |y_1(\omega, t) - y_2(\omega, t)|\,dt.
\end{split}
\end{equation}
The first term after ``$\le$'' sign is a non-decreasing function. The function $\frac{|y_1(\omega, x) - y_2(\omega, x)|}{x^\ell}$ is continuous on $[0,b]$ due to asymptotic condition \eqref{bl asympt}. Hence applying Gr\"{o}nwall's inequality we obtain that
\[
\left(\frac{b+|\omega| x}{x}\right)^{\ell+1}|y_1(\omega, x) - y_2(\omega, x)| \le C\tilde C_{q_1}|\omega|^{\ell+1} \int_0^x \frac{s|q_1(s) - q_2(s)|}{b+|\omega|s}\,ds \cdot  \exp\biggl( C \int_0^x \frac{t|q_2(t)|}{b+|\omega| t}\,dt\biggr),
\]
or
\begin{equation*}
    |y_1(\omega, x) - y_2(\omega, x)| \le C\tilde C_{q_1}\left(\frac{|\omega|x}{b+|\omega| x}\right)^{\ell+1} \exp\biggl( \frac{C}{b} \int_0^x t|q_2(t)|\,dt\biggr)\int_0^x \frac{s|q_1(s) - q_2(s)|}{b+|\omega|s}\,ds.
\end{equation*}

The proof for the case $\ell = -1/2$ is completely similar, the additional factor $1-\log\frac tb$ results in the change of $q_1$ and $q_2$ by $\tilde q_1$ and $\tilde q_2$. Combining we obtain the following estimate
\begin{equation}\label{Final ineq for diff}
    |y_1(\omega, x) - y_2(\omega, x)| \le C\tilde C_{q_1}\left(\frac{|\omega|x}{b+|\omega| x}\right)^{\ell+1} \exp\biggl( \frac{C}{b} \int_0^x t|\tilde q_2(t)|\,dt\biggr)\int_0^x \frac{s|\tilde q_1(s) - \tilde q_2(s)|}{b+|\omega|s}\,ds.
\end{equation}

Let us introduce the notation
\[
C(q, x) := C\exp\biggl(\frac Cb \int_0^x t|\tilde q|\,dt\biggl).
\]
\begin{theorem}\label{Thm continuous dep}
Let the potentials $q_1$ and $q_2$ satisfy condition \eqref{Cond on q} with the same exponent $\mu$ and let $y_{1,2}$ be regular solutions of \eqref{Eq Bessel ContDep} satisfying asymptotics \eqref{bl asympt}. Then for each $x\in (0,b]$
\[
\|y_1(\cdot,x) - y_2(\cdot, x)\|_{L_2(\mathbb{R})} \le c_\mu C(q_1,x) C(q_2,x) \int_0^x t^\mu |\tilde q_1(t) - \tilde q_2(t)|\,dt,
\]
where the constant $c_\mu$ does not depend on the potentials.
\end{theorem}
\begin{proof}
It follows from \eqref{Final ineq for diff} that for $|\mu|>1$
\[
\begin{split}
|y_1(\omega, x) - y_2(\omega, x)|&\le C(q_1,x) C(q_2,x) \int_0^x t^\mu |\tilde q_1(t) - \tilde q_2(t)| \cdot \left(\frac{t}{b+|\omega|t}\right)^{1-\mu} \left(\frac{1}{b+|\omega|t}\right)^\mu dt\\
&\le C(q_1,x) C(q_2,x) \frac{1}{|\omega|^{1-\mu}}\frac{1}{ b^\mu} \int_0^x t^\mu |\tilde q_1(t) - \tilde q_2(t)|\,dt
\end{split}
\]
and for $|\mu|\le 1$
\[
|y_1(\omega, x) - y_2(\omega, x)|\le C(q_1,x) C(q_2,x) \left(\frac{x}{b}\right)^{1-\mu}\frac{1}{ b^\mu} \int_0^x t^\mu |\tilde q_1(t) - \tilde q_2(t)|\,dt.
\]
Combining the last two inequalities the statement follows.
\end{proof}

\begin{corollary}\label{Corr continuous dep K}
Let the potentials $q_1$ and $q_2$ satisfy condition \eqref{Cond on q} with the same exponent $\mu$ and let $K_{1,2}$ be corresponding integral kernels of the transmutation operators. Then for each $x\in (0,b]$
\[
\|K_1(x,\cdot) - K_2(x,\cdot)\|_{L_2(0,b)} \le \frac{c_\mu}{2} C(q_1,x) C(q_2,x) \int_0^x t^\mu |\tilde q_1(t) - \tilde q_2(t)|\,dt.
\]
\end{corollary}
\begin{proof}
Recall that for each fixed $x$ the integral kernels $K_1$ and $K_2$ are the Hankel transforms of the solutions $y_1$ and $y_2$ from Theorem \ref{Thm continuous dep} \cite[Theorem 2.4]{Sta2}. Hence
\[
K_1(x,t) - K_2(x,t) = \int_0^\infty \bigl(y_1(\omega, x) - y_2(\omega, x)\bigr) \sqrt{\omega t} J_{\ell+1/2}(\omega t)\,dt.
\]
Now the statement follows from Parseval's equality for the Hankel transform.
\end{proof}

Consider the function
\[
\nu(x) = \begin{cases}
x^\mu, & \ell>-1/2,\\
x^\mu (1-\log\frac xb), & \ell= -1/2.
\end{cases}
\]
Then potential $q$ satisfies condition \eqref{Cond on q} if and only if $q\in L_1\bigl((0,b),\nu(x)\,dx\bigr)$.

\begin{corollary}\label{Corr continuity K}
Let the potential $q$ satisfy condition \eqref{Cond on q} with a parameter $\mu$ and let $\{q_n\}_{n=0}^\infty\subset L_1\bigl((0,b),\nu(x)\,dx\bigr)$ be a sequence of potentials such that
\begin{equation}\label{qn to q}
q_n \overset{L_1((0,b),\nu(x)\,dx)}{\longrightarrow} q,\qquad n\to\infty.
\end{equation}
Let $K$, $\tilde R$ and $K_n$, $\tilde R_n$ be corresponding integral kernels. Then for any $x>0$
\[
K_n(x,\cdot) \overset{L_2(0,x)}{\longrightarrow} K(x,\cdot)\qquad\text{and}\qquad R_n(x,\cdot) \overset{L_2(-x,x)}{\longrightarrow} R(x,\cdot),\quad n\to\infty.
\]
\end{corollary}
\begin{proof}
It follows from \eqref{qn to q} that the quantities $C(q_n,x)$ are uniformly bounded by $2C(q,x)$ for all $n\ge n_0$. Hence the convergence
$K_n(x,\cdot) \to K(x,\cdot)$ follows immediately from Corollary \ref{Corr continuous dep K}.

Recall that the integral kernel $R$ is the Fourier transform of the regular solution $\tilde y$ satisfying the asymptotics $\tilde y(\omega, x) \sim x^{\ell+1}$ at 0. Dividing \eqref{Final ineq for diff} by $\omega^{\ell+1}$ we obtain that
\[
|\tilde y(\omega, x) - \tilde y_n(\omega,x)| \le C(q,x)C(q_n,x) \left(\frac{x}{b+|\omega|x}\right)^{\ell+1}\int_0^x \frac{s|\tilde q_1(s) - \tilde q_2(s)|}{b+|\omega|s}\,ds.
\]
Using the last inequality one can prove that $R_n(x,\cdot) \to R(x,\cdot)$ following the proof for the kernels $K$ and $K_n$ with minimal changes.
\end{proof}

\section{On existence of non-vanishing particular solution}\label{AppA}
Problem of existence of a non-vanishing particular solution of a differential equation naturally arises in the construction of coefficients for both SPPS \cite{KravchenkoPorter, CKT2013} and Neumann series of Bessel functions \cite{KNT 2015, KTC2017} representations for solutions. While for a non-singular case it is always possible to choose such linear combination of linearly independent particular solutions that it will not vanish on the whole interval of interest (see \cite[Remark 5]{KravchenkoPorter}, see also \cite{Camporesi et al 2011}), a perturbed Bessel equation possesses only one (up to a multiplicative constant) regular solution which either vanishes at some point or not. In this appendix we show that at least it is always possible to perform such spectral shift, i.e., consider an equation
\begin{equation}\label{EqHomShifted}
    -u''+\left(\frac{\ell(\ell+1)}{x^2}+q(x)+\lambda\right)u=0,
\end{equation}
that its regular solution will be non-vanishing on $(0,b]$.

In some situations such value of $\lambda$ is easy to choose. For example, if $q$ is real valued and bounded from below, any $\lambda\ge -\inf_{(0,b]} q(x)$ works, see, e.g., \cite[Corollary 3.3]{CKT2013}. For real valued potentials any $\lambda$ having $\operatorname{Im}\lambda\ne 0$ works since the equality $u(x_0)=0$ would imply the existence of imaginary eigenvalue for selfadjoint problem \eqref{EqHomShifted} with Dirichlet boundary conditions formulated on $[0,x_0]$, c.f., \cite[Remark 4.1]{CKT2015}. Lemma 3.1 from \cite{Carlson1993} shows that for real valued potentials satisfying $q\in L_2(0,b)$ any $\lambda$ satisfying $\lambda \ge \int_0^b |q(x)|^2\,dx$ works.

The following proposition shows the existence of such $\lambda$ for any potential satisfying
\begin{equation}\label{EqCondOnQ}
    \begin{split}
    xq(x) &\in L_1(0,b) \qquad \text{if } \ell>-1/2,\\
    xq(x)\Bigl(1-\log \frac xb\Bigr) &\in L_1(0,b)\qquad \text{if }\ell=-1/2.
    \end{split}
\end{equation}

\begin{proposition}\label{Prop NonVanishing Sol}
Let $q$ be complex valued function satisfying condition \eqref{EqCondOnQ}. Then there exist such constant $\lambda_0> 0$ that for all $\lambda>\lambda_0$ the regular solution of equation \eqref{EqHomShifted} does not vanish on $(0,b]$.
\end{proposition}

\begin{proof}
Let $\lambda = \omega^2$, $\omega >0$. The regular solution of \eqref{EqHomShifted} can be written in the form
\begin{equation*}
u(x,\omega) = \sqrt{\omega x} I_{\ell+1/2}(\omega x) + r(\omega, x),
\end{equation*}
where $I_{\ell+1/2}$ is the modified Bessel function of the first kind. The function $r(\omega, x)$ satisfies, see \cite[Lemma 2.2]{KosSakhTesh2010}
\begin{equation*}
    |r(\omega, x)| \le C\left(\frac{\omega x}{b+\omega x}\right)^{\ell+1} e^{\omega x} \int_0^x \frac{y\tilde q(y)}{b+\omega y}\,dy,
\end{equation*}
where the constant $C$ does not depend on $\omega$; $\tilde q(y) = |q(y)|$ if $\ell>-1/2$ and $\tilde q(y) = |q(y)|[1-\log(y/b)]$ if $\ell=-1/2$.

Due to the asymptotics $I_\nu(z)\sim \frac{e^z}{\sqrt{2\pi z}}$, $z\to\infty$ \cite[(9.7.1)]{Abramowitz}, there exists such $z_0$ that $I_{\ell+1/2}(z)\ge \frac{e^z}{4\sqrt z}$ for all $z\ge z_0$. Let $z_1:= \max\{z_0, 8C\int_0^b y\tilde q(y)\,dy\}$ (the integral is finite due to the condition \eqref{EqCondOnQ}) and let $x_0>0$ be such that $\int_0^{x_0} y\tilde q(y)\,dy<b/8C$.

If $\omega x\ge z_1$ and $\omega\ge z_1/x_0$, then
\[
|r(\omega, x)|<\frac{Ce^{\omega x}}{b}\int_0^{x_0}y\tilde q(y)\,dy +\frac{Ce^{\omega x}}{\omega x_0}\int_{x_0}^x y\tilde q(y)\,dy\le \frac{e^{\omega x}}4\le  \sqrt{\omega x}I_{\ell+1/2}(\omega x)
\]
and hence $u(x,\omega)\ne 0$.

Let $0<\omega x<z_1$. We may estimate the function $I_{\ell+1/2}$ by the first term of its Taylor series, i.e.,
\[
\sqrt{\omega x}I_{\ell+1/2}(\omega x) \ge \frac{(\omega x)^{\ell+1}}{2^{\ell+1/2}\Gamma(\ell+3/2)},\qquad 0<\omega x<z_1.
\]
Since $y\tilde q(y)\in L_1(0,b)$, there exists $x_1>0$ such that
\[
\int_0^{x_1} y\tilde q(y)\, dy \le \frac{b^{\ell+2}}{2^{\ell+1/2}\Gamma(\ell+3/2)Ce^{z_1}}.
\]
Take $\omega\ge z_1/x_1$. Then it follows from the inequality $\omega x<z_1$ that $x<x_1$. Hence
\[
|r(\omega, x)|<\frac{Ce^{z_1}(\omega x)^{\ell+1}}{b^{\ell+2}}\int_0^{x_1} y\tilde q(y)\, dy\le \sqrt{\omega x}I_{\ell+1/2}(\omega x),
\]
and so $u(x,\omega)\ne 0$.

Hence for all values of $\omega$ greater than $\max\bigl\{ \frac{z_1}{x_0}, \frac{z_1}{x_1}\bigr\}$ the function $u(x,\omega)$ does not vanish on $(0,b]$.
\end{proof}

\end{document}